\documentclass[11pt, twoside, leqno]{article}

\usepackage{amssymb}
\usepackage{amsmath}
\usepackage{amsthm}
\usepackage{amsfonts}
\usepackage{mathrsfs}
\usepackage{color}

\allowdisplaybreaks

\def\rr{{\mathbb R}}
\def\rn{{\mathbb{R}^n}}
\def\urn{\mathbb{R}_+^{n+1}}
\def\zz{{\mathbb Z}}
\def\cc{{\mathbb C}}

\def\nn{{\mathbb N}}

\def\ca{{\mathcal A}}
\def\cb{{\mathcal B}}

\def\cf{{\mathcal F}}

\def\cl{{\mathcal L}}
\def\cm{{\mathcal M}}

\def\cp{{\mathcal P}}

\def\cs{{\mathcal S}}
\def\ct{{\mathcal T}}
\def\cx{{\mathcal X}}
\def\cy{{\mathcal Y}}

\def\fz{\infty }
\def\az{\alpha}
\def\bz{\beta}

\def\ez{\epsilon}
\def\gz{{\gamma}}

\def\lz{\lambda}

\def\vz{\varphi}
\def\vez{\varepsilon}

\def\lf{\left}
\def\r{\right}

\def\hs{\hspace{0.25cm}}
\def\ls{\lesssim}

\def\noz{\nonumber}
\def\wz{\widetilde}
\def\wh{\widehat}

\def\st{\subset}

\def\loc{{\mathop\mathrm{\,loc\,}}}
\def\supp{\mathop\mathrm{\,supp\,}}

\def\vlp{{L^{p(\cdot)}(\rn)}}
\def\vhs{H^{p(\cdot)}(\rn)}

\def\hlp{H_L^{p(\cdot)}(\rn)}

\newtheorem{theorem}{Theorem}[section]
\newtheorem{lemma}[theorem]{Lemma}
\newtheorem{corollary}[theorem]{Corollary}
\newtheorem{proposition}[theorem]{Proposition}

\theoremstyle{definition}
\newtheorem{remark}[theorem]{Remark}
\newtheorem{definition}[theorem]{Definition}
\newtheorem{assumption}[theorem]{Assumption}

\renewcommand{\appendix}{\par
   \setcounter{section}{0}%
   \setcounter{subsection}{0}%
   \setcounter{subsubsection}{0}%
   \gdef\thesection{\@Alph\c@section}%
   \gdef\thesubsection{\@Alph\c@section.\@arabic\c@subsection}%
   \gdef\theHsection{\@Alph\c@section.}%
   \gdef\theHsubsection{\@Alph\c@section.\@arabic\c@subsection}%
   \csname appendixmore\endcsname
 }

\numberwithin{equation}{section}

\begin{document}

\arraycolsep=1pt

\title{\bf\Large Molecular Characterizations and Dualities of
Variable Exponent Hardy Spaces Associated with Operators
\footnotetext{\hspace{-0.35cm} 2010 {\it
Mathematics Subject Classification}. Primary 42B35;
Secondary 42B30, 35K08, 47D03.
\endgraf {\it Key words and phrases.} Hardy space, BMO space,
variable exponent, operator, heat kernel, molecule
\endgraf This project is supported by the National
Natural Science Foundation of China
(Grant Nos.~11571039 and 11361020),
the Specialized Research Fund for the Doctoral Program of Higher Education
of China (Grant No. 20120003110003) and the Fundamental Research
Funds for Central Universities of China
(Grant Nos.~2013YB60 and 2014KJJCA10). }}
\author{Dachun Yang and Ciqiang Zhuo\footnote{Corresponding author}}
\date{}
\maketitle

\vspace{-0.8cm}

\begin{center}
\begin{minipage}{13cm}
{\small {\bf Abstract}\quad
Let $L$ be a linear operator on $L^2(\mathbb R^n)$ generating
an analytic semigroup $\{e^{-tL}\}_{t\ge0}$ with kernels having pointwise upper bounds and
$p(\cdot):\ \mathbb R^n\to(0,1]$ be a variable exponent function satisfying
the globally log-H\"older continuous condition. In this article, the authors
introduce the variable exponent Hardy space associated with
the operator $L$, denoted by $H_L^{p(\cdot)}(\mathbb R^n)$, and
the BMO-type space ${\mathrm{BMO}}_{p(\cdot),L}(\mathbb R^n)$.
By means of tent spaces with variable exponents,
the authors then establish the molecular characterization of
$H_L^{p(\cdot)}(\mathbb R^n)$ and a duality theorem between such a Hardy space
and a BMO-type space. As applications, the authors study the boundedness
of the fractional integral on these Hardy spaces and the coincidence
between $H_L^{p(\cdot)}(\mathbb R^n)$ and
the variable exponent Hardy spaces $H^{p(\cdot)}(\mathbb R^n)$.
}
\end{minipage}
\end{center}

\vspace{-0.1cm}

\section{Introduction\label{s1}}
\hskip\parindent
In recent years, function spaces with variable exponents attract much attentions
(see, for example, \cite{ah10,cfbook,cw14,dhr11,dhr09,DDr15,ns12,
Sa13,yzy14,yzy15,zsy15,zyl14} and their references).
The variable exponent Lebesgue space $\vlp$, with an exponent function
$p(\cdot):\ \rn\to(0,\fz)$, which consists of all measurable functions $f$ such that
$\int_\rn |f(x)|^{p(x)}\,dx<\fz$, is a generalization of the classical Lebesgue space.
The study of variable exponent Lebesgue spaces can be traced back to
Birnbaum-Orlicz \cite{bo31} and Orlicz \cite{ol32}
(see also Luxemburg \cite{Lu55} and Nakano \cite{nak50,nak51}),
but the modern
development started with the articles \cite{kr91} of Kov\'a\v{c}ik
and R\'akosn\'{\i}k as well as \cite{cruz03} of Cruz-Uribe and \cite{din04} of Diening.
The variable function spaces have been
widely used in the study of harmonic analysis; see, for example,
\cite{cfbook,dhr11}. Apart from theoretical considerations,
such function spaces also have interesting applications in fluid dynamics
 \cite{am02,rm00}, image processing \cite{clr06},
partial differential equations and variational calculus
\cite{am05,hhl08,su09}.

Particularly, Nakai and Sawano \cite{ns12} introduced Hardy spaces with variable
exponents, $H^{p(\cdot)}(\rn)$, and established their atomic characterizations
which were further applied to consider dual spaces of such Hardy spaces.
Later, in \cite{Sa13}, Sawano extended the
atomic characterization of the space $H^{p(\cdot)}(\rn)$ in \cite{ns12},
which also improves the corresponding result in \cite{ns12},
and gave out more applications including the boundedness of the
fractional integral operator and the commutators generated by singular integral
operators and BMO functions, and an Olsen's inequality.
After that, Zhuo et al. \cite{zyl14} established their equivalent characterizations
via intrinsic square functions, including the intrinsic Lusin area function,
the intrinsic $g$-function and the intrinsic $g_\lz^\ast$-function. Independently,
Cruz-Uribe and Wang \cite{cw14} also investigated the variable exponent Hardy
space with some slightly weaker conditions than those used in \cite{ns12}.
Recall that the theory of classical Hardy spaces $H^p(\rn)$ with $p\in(0,1]$
and their duals are well studied and certainly play an important role in
harmonic analysis as well as partial differential equations;
see, for example, \cite{CoWe77,FeSt72,muller94,{stein93}}.

On the other hand, in recent years, the study of function spaces,
especially on Hardy spaces associated with different operators, has also inspired great interests
(see, for example, \cite{admun,dxy07,dy05,dy05cpam,jy10,jyz09,yan08,ls11}
and their references).
Particularly, let $L$ be a linear operator on $L^2(\mathbb R^n)$ and
generate an analytic semigroup $\{e^{-tL}\}_{t\ge0}$ with kernel having
pointwise upper bounds, whose decay is measured by $\theta(L)\in(0,\fz]$.
Then, by using the Lusin area function,
Auscher, Duong and McIntosh \cite{admun} initially introduced
the Hardy space $H_L^1(\rn)$ associated with the operator $L$ and established
its molecular characterization. Based on this, Duong and Yan \cite{dy05,dy05cpam}
introduced the BMO-type space BMO$_L(\rn)$ associated with $L$
and proved that the dual space of $H_L^1(\rn)$ is just BMO$_{L^\ast}(\rn)$,
where $L^\ast$ denotes the \emph{adjoint operator} of $L$ in $L^2(\rn)$.
Later, Yan \cite{yan08} further generalized these results to the Hardy spaces
$H_L^p(\rn)$ with $p\in(n/[n+\theta(L)],1]$ and their dual spaces.
Moreover, Jiang et al. \cite{jyz09} investigated the Orlicz-Hardy space
and its dual space associated with such an operator $L$.

Let $p(\cdot):\ \mathbb R^n\to(0,1]$ be a variable exponent function satisfying
the globally log-H\"older continuous condition.
Motivated by \cite{ns12,yan08}, in this article,
we introduce the variable exponent Hardy space associated with the operator $L$,
denoted by $H_L^{p(\cdot)}(\rn)$. More precisely,
for all $f\in L^2(\rn)$ and $x\in\rn$, let
\begin{equation*}
S_L(f)(x):=\lf\{\int_{\Gamma(x)}\lf|t^mLe^{-t^mL}(f)(y)
\r|^2\,\frac{dydt}{t^{n+1}}\r\}^{\frac12},
\end{equation*}
where $m$ is a positive constant appearing in the pointwise upper bound of the heat kernel
(see \eqref{assump1} below) and $\Gamma(x):=\{(y,t)\in\rn\times(0,\fz):\ |y-x|<t\}$.
The \emph{Hardy spaces $H_L^{p(\cdot)}(\rn)$} is defined to be the
completion of the set $\{f\in L^2(\rn):\ S_L(f)\in\vlp\}$
with respect to the quasi-norm
$$\|f\|_{H_L^{p(\cdot)}(\rn)}:=\|S_L(f)\|_{\vlp}
:=\inf\lf\{\lz\in(0,\fz):\ \int_\rn\lf[\frac{S_L(f)(x)}{\lz}\r]^{p(x)}\,dx\le1\r\}.$$
We then establish the molecular characterization of $H_L^{p(\cdot)}(\rn)$
via variable exponent tent spaces. Using this
molecular characterization, we further prove that the dual space of
$H_L^{p(\cdot)}(\rn)$ is the BMO-type space BMO$_{p(\cdot),L^\ast}(\rn)$, which is also
introduced in this article. As more applications, we study the boundedness
of the fractional integral $L^{-\gamma}$ ($\gamma\in(0,\frac nm)$ with $m$ as in
Assumption (A) below)
from $H_L^{p(\cdot)}(\rn)$ to $H_L^{q(\cdot)}(\rn)$ with
$\frac 1{q(\cdot)}:=\frac1{p(\cdot)}-\frac {m\gz}n$
and the coincidence between $H_L^{p(\cdot)}(\mathbb R^n)$
and variable exponent Hardy spaces $H^{p(\cdot)}(\rn)$ introduced in \cite{ns12}.

A novel aspect of this article is to give a non-trivial combination of function
spaces with variable exponents and the theory of operators including their functional
calculi and semigroups, and these new function spaces prove necessary in the study of the boundedness
of the associated operators (for example, fractional integrals $L^{-\gz}$ with $\gamma\in(0,\frac nm)$).

This article is organized as follows.

In Section \ref{s2},
we first recall some notation and definitions about variable exponent Lebesgue spaces,
holomorphic functional calculi of operators and semigroups,
also including some basic assumptions on the operator $L$ considered in this article
and the domain of the semigroup $\{e^{-tL}\}_{t\ge0}$. Via the Lusin area function $S_L(f)$,
we then introduce the variable exponent Hardy space associated with $L$, denoted by
$H_L^{p(\cdot)}(\rn)$.

In Section \ref{s3}, we mainly establish a molecular characterization
of the space $H_L^{p(\cdot)}(\rn)$ (see Theorem \ref{t-mol} below).
To this end, we first establish an atomic characterization
of the variable exponent tent space $T_2^{p(\cdot)}(\urn)$
(see Corollary \ref{c-tent} below). Then the
molecular characterization of $H_L^{p(\cdot)}(\rn)$ is obtained by using a
project operator
$\pi_L$ corresponding to $L$, which is proved to be bounded from
$T_2^{p(\cdot)}(\urn)$ to
$H_L^{p(\cdot)}(\rn)$. We point out that \cite[Lemma 4.1]{Sa13} of Sawano
(a slight weaker variant of this lemma was early obtained by Nakai and Sawano
\cite[Lemma 4.11]{ns12}),
which is re-stated in Lemma \ref{l-estimate} below, plays a key role
in the proof of Theorem \ref{t-mol}

Section \ref{s4} is devoted to proving a duality theorem.
Indeed, in Theorem \ref{t-dual} below, we show that the dual space
of $H_L^{p(\cdot)}(\rn)$ is just the BMO-type space BMO$_{p(\cdot),L^\ast}(\rn)$,
which is also introduced in this section.
To show Theorem \ref{t-dual}, we rely on several key estimates
related to BMO-type spaces and $p(\cdot)$-Carleson measures
(see Propositions \ref{p-bmo1}, \ref{p-bmo2} and \ref{p-carleson}, and Lemma
\ref{l-dual} below), and the duality of the variable exponent tent space
(see Proposition \ref{p-dual} below).
The main difficulty to establish these estimates is that
the quasi-norm $\|\cdot\|_{\vlp}$ has no the translation invariance,
namely, for any cube $Q(x,r)\st\rn$, with $x\in\rn$ and $r\in (0,\fz)$, and $z\in\rn$,
$\|\chi_{Q(x,r)}\|_{\vlp}$ may not equal to $\|\chi_{Q(x+z,r)}\|_{\vlp}$.
To overcome this difficulty, we make full use of Lemma \ref{l-bigsball} below, which
is just \cite[Lemma 2.6]{zyl14} and presents a relation between
two quasi-norms $\|\cdot\|_{\vlp}$ corresponding to two cubes.

As applications of the molecular characterization of $H_L^{p(\cdot)}(\rn)$
from Theorem \ref{t-mol}, in Section \ref{s5},
we investigate the boundedness of fractional integrals on $H_L^{p(\cdot)}(\rn)$
(see Theorem \ref{t-frac} below)
and show that the spaces $H_L^{p(\cdot)}(\rn)$ and $H^{p(\cdot)}(\rn)$ coincide
with equivalent quasi-norms under some additional assumptions on $L$
(see Theorem \ref{t-5.2x} below).

\section{Preliminaries\label{s2}}
\hskip\parindent
In this section, we first recall some notation and notions on
variable exponent Lebesgue spaces and some knowledge
about holomorphic functional calculi as well as semigroups.
Then we introduce the variable exponent Hardy spaces associated with
operators, denoted by $H_L^{p(\cdot)}(\rn)$, which generalize the Hardy spaces
$H_L^{p}(\rn)$ studied in \cite{dy05,yan08}.

We begin with some notation which will be used in this article.
Let $\nn:=\{1,2,\dots\}$ and $\zz_+:=\nn\cup\{0\}$.
We denote by $C$ a \emph{positive constant} which is independent of the main
parameters, but may vary from line to line. We use $C_{(\az,\dots)}$
to denote a positive constant depending on the indicated
parameters $\az,\, \dots$. The \emph{symbol}
$A\ls B$ means $A\le CB$. If $A\ls B$ and $B\ls A$, then we write $A\sim B$.
If $E$ is a subset of $\rn$, we denote by $\chi_E$ its
\emph{characteristic function} and by $E^\complement$ the set $\rn\backslash E$.
For $a \in {\mathbb R}$,
$\lfloor a \rfloor$ denotes the largest integer $m$
such that $m \le a$. For all $x\in\rn$ and $r\in(0,\fz)$, denote by $Q(x,r)$
the cube centered at $x$ with side length $r$, whose sides are parallel to the axes
of coordinates. For each cube $Q\st\rn$ and $a\in(0,\fz)$,
we use $x_Q$ to denote the center of $Q$ and $\ell(Q)$ to denote the side length
of $Q$, and denote by $aQ$
the cube concentric with $Q$ having the side length $a\ell(Q)$.

\subsection{Variable exponent Lebesgue spaces}
\hskip\parindent
In what follows, a measurable function $p(\cdot):\ \rn\to(0,\fz)$ is called a
\emph{variable exponent}. For any variable exponent $p(\cdot)$, let
\begin{equation}\label{2.1x}
p_-:=\mathop{\rm ess\,inf}\limits_{x\in \rn}p(x)
\quad {\rm and}\quad
p_+:=\mathop{\rm ess\,sup}\limits_{x\in \rn}p(x).
\end{equation}
Denote by $\cp(\rn)$ the \emph{collection of variable exponents}
$p(\cdot):\ \rn\to(0,\fz)$ \emph{satisfying} $0<p_-\le p_+<\fz$.

For a measurable function $f$ on $\rn$ and a variable exponent $p(\cdot)\in\cp(\rn)$,
the \emph{modular} $\varrho_{p(\cdot)}(f)$ of $f$
is defined by setting
$\varrho_{p(\cdot)}(f):=\int_\rn|f(x)|^{p(x)}\,dx$ and the
\emph{Luxemburg quasi-norm}
\begin{equation*}
\|f\|_{\vlp}:=\inf\lf\{\lz\in(0,\fz):\ \varrho_{p(\cdot)}(f/\lz)\le1\r\}.
\end{equation*}
Then the \emph{variable exponent Lebesgue space} $\vlp$ is defined to be the
set of all measurable functions $f$ such that $\varrho_{p(\cdot)}(f)<\fz$
equipped with the quasi-norm $\|f\|_{\vlp}$. For more properties on the variable
exponent Lebesgue spaces, we refer the reader to \cite{cfbook,{dhr11}}.

\begin{remark}\label{r-vlp}
 Let $p(\cdot)\in\cp(\rn)$.

(i) If $p_-\in[1,\fz)$, then $L^{p(\cdot)}(\rn)$
is a Banach space (see \cite[Theorem 3.2.7]{dhr11}).
In particular, for all $\lz\in\cc$ and $f\in\vlp$,
$\|\lz f\|_{\vlp}=|\lz|\|f\|_{\vlp}$ and, for all $f,\ g\in\vlp$,
$$\|f+g\|_{\vlp}\le \|f\|_{\vlp}+\|g\|_{\vlp}.$$

(ii) For any non-trivial function $f\in \vlp$, it holds true that
$\varrho(f/\|f\|_{\vlp})=1$; see, for example, \cite[Proposition 2.21]{cfbook}.

(iii) If $\int_\rn[|f(x)|/\delta]^{p(x)}\,dx\le c$ for some $\delta\in(0,\fz)$
and some positive constant $c$ independent of $\delta$, then it is easy to see that
$\|f\|_{\vlp}\le C\delta$, where $C$ is a positive constant independent of $\delta$,
but depending on $p_-$ (or $p_+$) and $c$.
\end{remark}

Recall that a measurable function $g\in\cp(\rn)$ is said to be
\emph{locally {\rm log}-H\"older continuous},
denoted by $g\in  C_{\rm loc}^{\log}(\rn)$,
if there exists a positive constant $C_{\log}(g)$ such that, for all $x,\ y\in\rn$,
\begin{equation*}
|g(x)-g(y)|\le \frac{C_{\log}(g)}{\log(e+1/|x-y|)},
\end{equation*}
and $g$ is said to satisfy the
\emph{globally {\rm log}-H\"older continuous condition},
denoted by $g\in  C^{\log}(\rn)$,
if $g\in  C_{\rm loc}^{\log}(\rn)$ and there exist a positive constant
$C_\fz$ and a constant $g_\fz\in\rr$
such that, for all $x\in\rn$,
\begin{equation*}
|g(x)-g_\fz|\le \frac{C_\fz}{\log(e+|x|)}.
\end{equation*}

\begin{remark}
Let $n=1$ and, for all $x\in\rr$,
$$p(x):=\max\lf\{1-e^{3-|x|},\min\lf(6/5,\max\lf\{1/2,3/2-x^2\r\}\r)\r\}.$$
Then $p(\cdot)\in C^{\log}(\rr)$; see \cite[Example 1.3]{ns12}.
With a slight modification,
another example was obtained in \cite[Example 2.20]{yyz14}
as follows. For all $x\in\rr$, let
$$ p(x):=\max\lf\{1-e^{3-|x|},\ \min\lf(6/5,\max(1/2,k|x|+1/2-k)\r)\r\},$$
where $k:=7/[10(\sqrt{3/10}-1)]$.
Then $p(\cdot)\in C^{\log}(\rr)$.
\end{remark}

For all $r\in(0,\fz)$, denote by $L_\loc^r(\rn)$ the \emph{set of all locally
$r$-integrable functions} on $\rn$ and, for any measurable set $E\st \rn$,
by $L^r(E)$ the \emph{set of all measurable functions $f$ such that}
$$\|f\|_{L^r(E)}:=\lf\{\int_E|f(x)|^r\,dx\r\}^{1/r}<\fz.$$
Recall that the \emph{Hardy-Littlewood maximal operator $\cm$} is defined by setting,
for all $f\in L_\loc^1(\rn)$ and $x\in\rn$,
$$\cm(f)(x):=\sup_{B\ni x}\frac1{|B|}\int_B|f(y)|\,dy,$$
where the supremum is taken over all balls $B$ of $\rn$ containing $x$.

\begin{remark}
Let $p(\cdot)\in C^{\log}(\rn)$ and $1<p_-\le p_+<\fz$. Then there exists a positive
constant $C$ such that, for all $f\in \vlp$,
$\|\cm(f)\|_{\vlp}\le C\|f\|_{\vlp}$;
see, for example, \cite[Theorem 4.3.8]{dhr11}.
\end{remark}

\subsection{Holomorphic functional calculi\label{s2.2}}
\hskip\parindent
Here, we first recall some notions of the bounded holomorphic functional calculus,
which were introduced by McIntosh \cite{mc86}, and then make two assumptions on $L$
required in this article.
For two normed linear spaces $\cx$ and $\cy$,
let $\cl(\cx,\cy)$ be the \emph{collection of continuous linear
operators from $\cx$ to $\cy$} and, for any $T\in \cl(\cx,\cy)$,
$\|T\|_{\cx\to\cy}$ its operator norm.

Let $v\in(0,\pi)$. Define the closed sector $S_v$ by
$S_v:=\{z\in\cc:\ |\arg z|\le v\}\cup\{0\}$ and denote by $S_v^0$ the
interior of $S_v$. Let $H(S_v^0)$ be the set of all holomorphic functions on
$S_v^0$,
$$H^\fz(S_v^0):=\lf\{b\in H(S_v^0):\ \|b\|_{\fz}:=\sup_{z\in S_v^0}|b(z)|<\fz\r\}$$
and
\begin{eqnarray*}
\Psi(S_v^0)&&:=\lf\{\psi\in H(S_v^0):\ \exists\ s,\, C\in(0,\fz)\
{\rm such\ that}\r.\\
&&\quad\quad\quad\quad\lf. |\psi(z)|\le C|z|^s(1+|z|^{2s})^{-1},\ \forall\ z\in S_v^0\r\}.
\end{eqnarray*}

Given $v\in(0,\pi)$, a closed operator $L\in\cl(L^2(\rn),L^2(\rn))$ is said to be of
\emph{type $v$} if $\sigma(L)\subset S_v$, where $\sigma(L)$ denotes the spectra of $L$,
and, for all $\gamma\in(v,\pi)$, there exists a positive constant $C$ such that,
for all $\lambda\notin S_\gamma$,
$$\|(L-\lz I)^{-1}\|_{L^2(\rn)\to L^2(\rn)}\le C|\lz|^{-1}.$$
Let $\theta\in(v,\gamma)$ and $\Sigma$ be the contour
$\{\xi=re^{\pm i\theta}: r\in[0,\fz)\}$ parameterized clockwise around $S_v$.
Then, for $\psi\in \Psi(S_v^0)$ and $L$ being of type $v$, the operator $\psi(L)$ is
defined by
$$\psi(L):=\frac1{2\pi i}\int_\Sigma(L-\lz I)^{-1}\psi(\lz)\,d\lz,$$
where the integral is absolutely convergent in $\cl(L^2(\rn),L^2(\rn))$ and, by
the Cauchy theorem, the above definition is independent
of the choices of $v$ and $\gamma$ satisfying $\theta\in(v,\gamma)$.
If $L$ is a one-to-one linear operator having dense range and $b\in H^\fz(S_\gamma^0)$,
then define an operator $b(L)$ by $b(L):=[\psi(L)]^{-1}(b\psi)(L)$,
where $\psi(z):=z(1+z)^{-2}$ for all $z\in S_{\gamma}^0$.
It was proved in \cite{mc86} that $b(L)$ is well defined on $L^2(\rn)$.
The operator $L$ is said to have a \emph{bounded $H^\fz$ functional calculus} on
$L^2(\rn)$ if, for all
$\gamma\in(v,\pi)$, there exists a positive constant $\wz C$ such that, for all
$b\in H^\fz(S_\gamma^0)$, $b(L)\in \cl(L^2(\rn),L^2(\rn))$ and
$$\|b(L)\|_{L^2(\rn)\to L^2(\rn)}\le \wz C\|b\|_\fz.$$

Let $L$ be a linear operator of type $v$ on $L^2(\rn)$ with $v\in(0,\frac \pi2)$.
Then it generates a bounded holomorphic semigroup $\{e^{-zL}\}_{z\in D_v}$, where
$D_v:=\{z\in\cc:\ 0\le|\arg(z)|<\frac\pi2-v\}$ and, for all $z\in\cc$,
$\arg(z)\in(-\pi,\pi]$ is the argument of $z$; see, for example, \cite[Theorem 1.45]{ou05}.

In this article, we make the following two assumptions on the operator $L$.

\setcounter{theorem}{0}
\renewcommand{\thetheorem}{(\Alph{theorem})}

\begin{assumption}\label{as-a}
Assume that, for each $t\in(0,\fz)$, the distribution kernel
$p_t$ of $e^{-tL}$ belongs to $L^\fz(\rn\times\rn)$ and satisfies that, for all $x,\ y\in\rn$,
\begin{equation}\label{assump1}
|p_t(x,y)|\le t^{-\frac nm}g\lf(\frac{|x-y|}{t^{\frac1m}}\r),
\end{equation}
where $m$ is a positive constant and $g$ is a positive, bounded and decreasing function
satisfying that, for some $\vez\in(0,\fz)$,
\begin{equation}\label{assump2}
\lim_{r\to\fz}r^{n+\vez}g(r)=0.
\end{equation}
\end{assumption}

\begin{assumption}\label{as-b} Assume that the operator $L$ is one-to-one, has dense range in
$L^2(\rn)$ and a bounded $H^\fz$ functional calculus on $L^2(\rn)$.
\end{assumption}

\setcounter{theorem}{3}
\renewcommand{\thetheorem}{\arabic{section}.\arabic{theorem}}

\begin{remark}
(i) If $\{e^{-tL}\}_{t\ge0}$ is a bounded analytic semigroup on $L^2(\rn)$ whose
kernels $\{p_t\}_{t\ge0}$ satisfy \eqref{assump1} and \eqref{assump2}, then, for
any $k\in\nn$, there exists a positive constant $C_{(k)}$, depending on $k$, such that, for all $t\in(0,\fz)$
and almost every $x,\ y\in\rn$,
\begin{equation}\label{partial1}
\lf|t^k\frac{\partial^kp_t(x,y)}{\partial t^k}\r|\le \frac {C_{(k)}}{t^{n/m}}
g_{k}\lf(\frac{|x-y|}{t^{1/m}}\r).
\end{equation}
Here, it should be pointed out that, for all $k\in\nn$, the function $g_{k}$ may
depend on $k$ but always satisfies \eqref{assump2};
see \cite[Theorem 6.17]{ou05} and \cite{cdu00}.

(ii) Let $v\in(0,\pi)$. Then $L$ has a bounded $H^\fz$ functional calculus on $L^2(\rn)$ if and only if,
for any $\gamma\in(v,\pi)$ and nonzero function $\psi\in \Psi(S_\gamma^0)$,
$L$ satisfies the following square function estimate:
there exists a positive constant
$C$ such that, for all $f\in L^2(\rn)$,
\begin{equation*}
C^{-1}\|f\|_{L^2(\rn)}\le\lf\{\int_0^\fz\|\psi_t(L)f\|_{L^2(\rn)}^2\,
\frac{dt}t\r\}^{1/2}
\le C\|f\|_{L^2(\rn)},
\end{equation*}
where $\psi_t(\xi):=\psi(t\xi)$ for all $t\in(0,\fz)$ and $\xi\in\rn$;
see \cite{mc86}.
\end{remark}

\subsection{An acting class of semigroups $\{e^{-tL}\}_{t\ge0}$}
\hskip\parindent
For all $\beta\in(0,\fz)$, let $\cm_{\beta}(\rn)$ be the
\emph{set of all functions $f\in L^2_{\loc}(\rn)$ satisfying}
$$\|f\|_{\cm_\beta(\rn)}:=\lf\{\int_\rn\frac{|f(x)|^2}
{1+|x|^{n+\beta}}\,dx\r\}^{1/2}<\fz.$$
We point out that the space $\cm_\beta(\rn)$ was introduced by Duong and Yan in
\cite{dy05cpam} and it is a Banach space under the norm
$\|\cdot\|_{\cm_\beta(\rn)}$. For any given operator $L$ satisfying
Assumptions \ref{as-a} and \ref{as-b}, let
\begin{equation}\label{2.3x}
\theta(L)
:=\sup\{\vez\in(0,\fz):\ (\ref{assump1})\ {\rm and}\ (\ref{assump2}){\ \rm hold\ true}\}
\end{equation}
and
\begin{eqnarray*}
\cm(\rn):=
\lf\{
\begin{array}{ll}
\cm_{\theta(L)}(\rn),\quad &{\rm if}\ \theta(L)<\fz,\\
\bigcup_{\beta\in(0,\fz)}\cm_\beta(\rn),\quad &{\rm if}\ \theta(L)=\fz.
\end{array}\r.
\end{eqnarray*}

Let $s\in\zz_+$. For any $f\in\cm(\rn)$ and $(x,t)\in\urn:=\rn\times(0,\fz)$, let
\begin{equation}\label{operator-pq}
P_{s,t}f(x):=f(x)-(I-e^{-tL})^{s+1}f(x)\quad{\rm and}\quad
Q_{s,t}f(x):=(tL)^{s+1}e^{-tL}f(x),
\end{equation}
and, particularly, let
\begin{equation}\label{operator-pqx}
P_tf(x):=P_{0,t}f(x)=e^{-tL}f(x)\quad{\rm and}\quad
Q_tf(x):=Q_{0,t}f(x)=tLe^{-tL}f(x).
\end{equation}
Here, we point out that these operators in \eqref{operator-pq} were
introduced by Blunck and Kunstmann \cite{bk04} and Holfmann and Martell \cite{hm03}.

\begin{remark}\label{r-operator}
(i) For all $f\in\cm(\rn)$, the operators $P_{s,t}f$ and $Q_{s,t}f$ are well defined. Moreover, the kernels $p_{s,t}$ of $P_{s,t}$ and $q_{s,t}$ of $Q_{s,t}$ satisfy that there exists a positive constant $C$ such that, for all $t\in(0,\fz)$ and $x,\ y\in\rn$,
\begin{equation}\label{operator-k}
|p_{s,t^m}(x,y)|+|q_{s,t^m}(x,y)|\le Ct^{-n}g\lf(\frac{|x-y|}{t}\r),
\end{equation}
where the function $g$ satisfies the conditions as in Assumption \ref{as-a}; see, for example, \cite{yan08}.

(ii)
A typical example of $L$ satisfying $\theta(L)=\fz$ is that the
kernels $\{p_t\}_{t\ge0}$ of $\{e^{-tL}\}_{t\ge0}$ have the pointwise
Gaussian upper bound, namely, there exists a positive
constant $C$ such that, for all $t\in(0,\fz)$ and $x,\ y\in\rn$,
$|p_t(x,y)|\le \frac C{t^{n/2}}e^{-\frac{|x-y|^2}{t}}$.
Obviously, if $\Delta:=\sum_{i=1}^n\frac{\partial^2}{\partial x_i}$ is the Laplacian
operator and $L=-\Delta$, then the heat kernels
have the pointwise Gaussian upper bound.
There are several other operators whose heat kernels have the pointwise
Gaussian upper bound; see, for example, \cite[p.\,4390, Remarks]{yan08}.

(iii) Let $s\in\zz_+$ and $p\in(1,\fz)$. Then, by (i) of this remark,
we easily conclude that there exists a positive constant $C$ such that,
for all $t\in(0,\fz)$ and $f\in L^p(\rn)$,
$$\|P_{s,t^m}(f)\|_{L^p(\rn)}\le C\|f\|_{L^p(\rn)}.$$
\end{remark}

\subsection{Definition of Hardy spaces $H_L^{p(\cdot)}(\rn)$}
\hskip\parindent
For all functions $f\in L^2(\rn)$, define the \emph{Lusin area function} $S_L(f)$
by setting, for all $x\in\rn$,
\begin{equation*}
S_L(f)(x):=\lf\{\int_{\Gamma(x)}\lf|Q_{t^m}f(y)\r|^2\,
\frac{dy\,dt}{t^{n+1}}\r\}^{1/2},
\end{equation*}
here and hereafter, for all $x\in\rn$,
$\Gamma(x):=\{(y,t)\in\rr_+^{n+1}:\ |y-x|<t\}$ and $Q_t$ is defined as in
\eqref{operator-pqx}.
In \cite{admun}, Auscher et al. proved that, for any $p\in(1,\fz)$,
there exists a positive
constant $C_{(p)}$, depending on $p$, such that, for all $f\in L^p(\rn)$,
\begin{equation}\label{SL-bounded}
C_{(p)}^{-1}\|f\|_{L^p(\rn)}\le \|S_L(f)\|_{L^p(\rn)}\le C_{(p)}\|f\|_{L^p(\rn)};
\end{equation}
see also Duong and McIntosh \cite{dm99} and Yan \cite{yan04}.

We now introduce the variable exponent Hardy spaces associated with operators.

\begin{definition}\label{d-hardys}
Let $L$ be an operator satisfying Assumptions \ref{as-a} and \ref{as-b}, and
$p(\cdot)\in C^{\log}(\rn)$ satisfy $p_+\in(0,1]$ .
A function $f\in L^2(\rn)$ is said to be in $\wz H_L^{p(\cdot)}(\rn)$ if
$S_L(f)\in L^{p(\cdot)}(\rn)$; moreover, define
$$\|f\|_{H_L^{p(\cdot)}(\rn)}
:=\|S_L(f)\|_{L^{p(\cdot)}(\rn)}
:=\inf\lf\{\lz\in(0,\fz):\ \int_\rn\lf[\frac{S_L(f)(x)}{\lz}\r]^{p(x)}\,dx\le1\r\}.$$
Then the \emph{variable Hardy space associated with operator} $L$,
denoted by $H_L^{p(\cdot)}(\rn)$,
is defined to be the completion of $\wz H_L^{p(\cdot)}(\rn)$ in the quasi-norm
$\|\cdot\|_{\hlp}$.
\end{definition}

\begin{remark}
(i) By the theorem of completion of Yosida \cite[p.\,65]{Yo78},
we find that $\wz H_L^{p(\cdot)}(\rn)$ is dense in $H_L^{p(\cdot)}(\rn)$,
namely, for any $f\in H_L^{p(\cdot)}(\rn)$, there exists a Cauchy sequence
$\{f_k\}_{k\in\nn}$ in $\wz H_L^{p(\cdot)}(\rn)$ such that
$\lim_{k\to\fz}\|f_k-f\|_{H_L^{p(\cdot)}(\rn)}=0$. Moreover,
if $\{f_k\}_{k\in\nn}$ is a Cauchy sequence in $\wz H_L^{p(\cdot)}(\rn)$, then
there exists an unique $f\in H_L^{p(\cdot)}(\rn)$ such that
$\lim_{k\to\fz}\|f_k-f\|_{H_L^{p(\cdot)}(\rn)}=0$.
Moreover, $L^2(\rn)\cap H_L^{p(\cdot)}(\rn)$ is dense in $H_L^{p(\cdot)}(\rn)$.

(ii) We point out that smooth functions with compact supports do not necessarily
belong to $H_L^{p(\cdot)}(\rn)$; see \cite{yan08} and also Remark \ref{r-dual} below
for more details.

(iii) Observe that, when $p(\cdot)\equiv p\in(0,\fz)$, $\vlp=L^p(\rn)$.
If $p(\cdot)\equiv 1$, then $\hlp=H_L^1(\rn)$, which was introduced by Auscher et al.
\cite{admun}; see also Duong and Yan \cite{dy05}.
If $p(\cdot)\equiv p\in (\frac n{n+\theta(L)},1)$, then
the space $\hlp$ is just the space $H_L^p(\rn)$ introduced by Yan \cite{yan08}.

(iv) Different from the space $H_L^p(\rn)$ which is just $L^p(\rn)$
when $p\in (1,\fz)$ (see, for example, \cite[p.\,4400]{yan08}),
since it is still unclear whether \eqref{SL-bounded} holds true or not with
$L^p(\rn)$ replaced by $\vlp$ when $p_+\in (1,\fz)$,  it is also
unclear whether $H^{p(\cdot)}_L(\rn)$ and $\vlp$ (or $\vhs$) coincide or not.
We will not push this issue in this article due to its length.
\end{remark}

We end this section by comparing the variable exponent Hardy spaces associated with
operators in this article with the Musielak-Orlicz-Hardy spaces
associated with operators satisfying reinforced off-diagonal estimates
in \cite{bckyy}. Indeed, in general, these two scales of Hardy-type spaces
do not cover each other.

\begin{remark}
Let $\vz:\ \rn\times[0,\fz)\to[0,\fz)$ be a growth function in \cite{ky14} and $L$
an operator satisfying reinforced off-diagonal estimates in \cite{bckyy}.
Then Bui et al. \cite{bckyy} introduced the Musielak-Orlicz-Hardy space
associated with operator $L$ via the Lusin area function,
denoted by $H_{\vz,L}(\rn)$.
Recall that the Musielak-Orlicz space $L^{\vz}(\rn)$ is defined to be the set
of all measurable functions $f$ on $\rn$ such that
$$\|f\|_{L^{\vz}(\rn)}:=\inf\lf\{\lz\in(0,\fz):\
\int_\rn\vz(x,|f(x)|/\lz)\,dx\le1\r\}<\fz.$$
Observe that, if
\begin{equation}\label{vz}
\vz(x,t):=t^{p(x)}\quad {\rm for\ all}\quad x\in\rn\quad{\rm and}\quad t\in[0,\fz),
\end{equation}
then $L^{\vz}(\rn)=L^{p(\cdot)}(\rn)$. However, a general
Musielak-Orlicz function $\vz$ satisfying all the assumptions
in \cite{ky14} (and hence \cite{bckyy}) may not have the form
as in \eqref{vz} (see \cite{ky14}). On the other hand, it was proved in
\cite[Remark 2.23(iii)]{yyz14} that there exists a variable exponent function
$p(\cdot)\in C^{\log}(\rn)$,
but $t^{p(\cdot)}$ is not a uniformly Muckenhoupt weight,
which was required in \cite{ky14} (and hence \cite{bckyy}).
Thus, Musielak-Orlicz-Hardy spaces associated with operators in \cite{bckyy}
and variable exponent Hardy spaces associated with operators in this article
do not cover each other.

Moreover, in Theorem \ref{t-5.2x} below, we show that, under some additional
assumptions on $L$, the spaces $H_L^{p(\cdot)}(\rn)$ coincide with
the variable exponent Hardy spaces $\vhs$ which can not cover and also can not
be covered by Musielak-Orlicz Hardy spaces in \cite{ky14} based on
the same reason as above.
\end{remark}

\section{Molecular characterizations of $H_L^{p(\cdot)}(\rn)$\label{s3}}
\hskip\parindent
In this section, we aim to obtain the molecular characterizations of
$H_L^{p(\cdot)}(\rn)$. To this end, we first give out some properties of the tent
spaces with variable exponents including their atomic characterizations, which are then
applied to establish the molecular characterizations of $H_L^{p(\cdot)}(\rn)$
by using a project operator $\pi_L$ corresponding to $L$.

\subsection{Atomic characterizations of tent spaces $T_2^{p(\cdot)}(\urn)$}
\hskip\parindent
We begin with the definition of the tent space with variable exponent.

Let $p(\cdot)\in\cp(\rn)$.
For all measurable functions $g$
on $\rr_+^{n+1}$ and $x\in\rn$, define
$$\ct(g)(x)
:=\lf\{\int_{\Gamma(x)}|g(y,t)|^2\,
\frac{dy\,dt}{t^{n+1}}\r\}^{1/2}$$
and
$$\mathcal{C}_{p(\cdot)}(g)(x):=\sup_{Q\ni x}\frac{|Q|^{1/2}}{\|\chi_Q\|_{\vlp}}
\lf\{\int_{\wh Q}|g(y,t)|^2\,\frac{dydt}{t}\r\}^{1/2},$$
where the supremum is taken over all cubes $Q$ of $\rn$ containing $x$
and $\wh Q$ denotes the tent over $Q$, namely, $\wh Q:=\{(y,t)\in \urn:\ B(y,t)\st Q\}$.

\begin{definition}
Let $p(\cdot)\in \cp(\rn)$.

(i) Let $q\in(0,\fz)$. Then the \emph{tent space} $T_2^q(\rr_+^{n+1})$
is defined to be the set of all measurable functions $g$ on $\urn$ such that
$\|g\|_{T_2^q(\rr_+^{n+1})}:=\|\ct(g)\|_{L^q(\rn)}<\fz.$

(ii) The \emph{tent space with variable exponent}
$T_2^{p(\cdot)}(\urn)$ is defined to be the set of all
measurable functions $g$ on $\rr_+^{n+1}$ such that
$\|g\|_{T^{p(\cdot)}_2(\rr_+^{n+1})}
:=\|\ct(g)\|_{\vlp}<\fz$.

(iii) The \emph{space} $T_{2,\fz}^{p(\cdot)}(\urn)$ is defined
to be the set of all measurable functions $g$ on $\rr_+^{n+1}$ such that
$\|g\|_{T_{2,\fz}^{p(\cdot)}(\urn)}:=\|\mathcal{C}_{p(\cdot)}(g)\|_{L^\fz(\rn)}<\fz.$
\end{definition}

\begin{remark}\label{r-atom}
(i) We point out that the spaces $T_2^q(\rr_+^{n+1})$
and $T^{p(\cdot)}_2(\rr_+^{n+1})$ were introduced in \cite{cms85}
and \cite{zyl14}, respectively. Moreover,
if $p(\cdot)\equiv q\in(0,\fz)$, then
$T^{p(\cdot)}_2(\rr_+^{n+1})=T_2^q(\rr_+^{n+1})$.

(ii) If $g\in T_2^2(\urn)$, then we easily see that
$\|g\|_{T_2^2(\urn)}=\{\int_{\urn}|g(x,t)|^2\,\frac{dxdt}t\}^{\frac12}$.
\end{remark}

Let $q\in(1,\fz)$ and $p(\cdot)\in\cp(\rn)$.
Recall that a measurable function $a$ on $\rr_+^{n+1}$ is called a
$(p(\cdot),q)$-\emph{atom} if $a$ satisfies

(i)$\supp a\subset\wh Q$ for some cube $Q\subset\rn$;

(ii)$\|a\|_{T_2^q(\rr_+^{n+1})}\le |Q|^{1/q}\|\chi_Q\|_{\vlp}^{-1}$.

Furthermore, if $a$ is a $(p(\cdot),q)$-atom for all
$q\in(1,\fz)$, then $a$ is call a $(p(\cdot),\fz)$-\emph{atom}. We point
out that the $(p(\cdot),\fz)$-atom was introduced in \cite{zyl14}.

For any $p(\cdot)\in\cp(\rn)$, $\{\lz_j\}_{j\in\nn}\st\cc$
and cubes $\{Q_j\}_{j\in\nn}$ of $\rn$, let
\begin{equation}\label{3.1y}
\ca\lf(\{\lz_j\}_{j\in\nn},\{Q_j\}_{j\in\nn}\r)
:=\lf\|\lf\{\sum_{j\in\nn}\lf[\frac{|\lz_j|\chi_{Q_j}}{\|Q_j\|_{\vlp}}
\r]^{\underline{p}}\r\}^{\frac 1{\underline{p}}}\r\|_{\vlp},
\end{equation}
here and hereafter, we let
\begin{equation}\label{3.1x}
\underline{p}:=\min\{1,p_-\}
\end{equation}
with $p_-$ as in \eqref{2.1x}.

The following atomic decomposition of $T_2^{p(\cdot)}(\urn)$ was proved in
\cite[Theorem 2.16]{zyl14}.
\begin{lemma}\label{l-tent}
Let $p(\cdot)\in C^{\log}(\rn)$. Then, for all
$f\in T^{p(\cdot)}_2(\urn)$, there exist $\{\lz_j\}_{j\in\nn}\subset\cc$
and a sequence $\{a_j\}_{j\in\nn}$ of $(p(\cdot),\fz)$-atoms such that,
for almost every $(x,t)\in\urn$,
\begin{equation}\label{tent2}
f(x,t)=\sum_{j\in\nn}\lz_ja_j(x,t);
\end{equation}
moreover, the series in \eqref{tent2} converges absolutely for almost
all $(x,t)\in\urn$ and there exists a positive constant $C$ such that, for all
$f\in T^{p(\cdot)}_2(\urn)$,
\begin{equation}\label{tent-y}
\cb(\{\lz_ja_j\}_{j\in\nn}):=\ca(\{\lz_j\}_{j\in\nn},\{Q_j\}_{j\in\nn})
\le C\|f\|_{T^{p(\cdot)}_2(\urn)},
\end{equation}
where, for each $j\in\nn$, $Q_j$ denotes the cube such that $\supp a_j\st \wh Q_j$.
\end{lemma}

By Lemma \ref{l-tent}, we have the following conclusion.

\begin{corollary}\label{c-tent1}
Let $p(\cdot)\in C^{\log}(\rn)$.
Assume that $f\in T_2^{p(\cdot)}(\urn)$, then the decomposition
\eqref{tent2} also holds true in $T_2^{p(\cdot)}(\urn)$.
\end{corollary}

To prove Corollary \ref{c-tent1}, we need the following useful lemma, which is just
\cite[Lemma 4.1]{Sa13}.

\begin{lemma}\label{l-estimate}
Let $p(\cdot)\in C^{\log}(\rn)$ and
$q\in[1,\fz)\cap(p_+,\fz)$ with $p_+$ as in \eqref{2.1x}.
Then there exists a positive constant
$C$ such that, for all sequences $\{Q_j\}_{j\in\nn}$ of cubes of $\rn$,
numbers $\{\lz_j\}_{j\in\nn}\st\cc$ and functions
$\{a_j\}_{j\in\nn}$ satisfying that,  for each $j\in\nn$,
$\supp a_j\st Q_j$ and $\|a_j\|_{L^q(\rn)}\le |Q_j|^{1/q}$,
\begin{equation*}
\lf\|\lf(\sum_{j=1}^\fz|\lz_ja_j|^{\underline{p}}\r)^\frac1{\underline{p}}\r\|_{\vlp}
\le C\lf\|\lf(\sum_{j=1}^\fz|\lz_j\chi_{Q_j}|^{\underline{p}}
\r)^\frac1{\underline{p}}\r\|_{\vlp},
\end{equation*}
where $\underline{p}$ is as in \eqref{3.1x}.
\end{lemma}

\begin{proof}[Proof of Corollary \ref{c-tent1}]
Let $f\in T_2^{p(\cdot)}(\urn)$.
Then, by Lemma \ref{l-tent}, we may assume that $f=\sum_{j\in\nn}\lz_ja_j$
almost everywhere on $\urn$, where $\{\lz_j\}_{j\in\nn}\subset\cc$
and $\{a_j\}_{j\in\nn}$ is a sequence of $(p(\cdot),\fz)$-atoms such that, for each
$j\in\nn$, $\supp a_j\subset \wh Q_j$ with some cube $Q_j\subset\rn$, and
\begin{equation}\label{tent6}
\ca\lf(\{\lz_j\}_{j\in\nn},\{Q_j\}_{j\in\nn}\r)\ls\|f\|_{T_2^{p(\cdot)}(\urn)}.
\end{equation}
Let $q\in[1,\fz)\cap(p_+,\fz)$. Then, by the definition of $(p(\cdot),\fz)$-atoms,
we see that, for all $j\in\nn$,
$$\|\ct(a_j)\|_{L^q(\rn)}=\|a_j\|_{T_2^q(\urn)}\le \frac{|Q_j|^{1/q}}
{\|\chi_{Q_j}\|_{\vlp}}.$$
From this, Lemma \ref{l-estimate} and the fact that, for all $\theta\in(0,1]$ and
$\{\xi_j\}_{j\in\nn}\st\cc$,
\begin{equation}\label{simple-ineq}
\lf(\sum_{j\in\nn}|\xi_j|\r)^{\theta}\le \sum_{j\in\nn}|\xi_j|^\theta,
\end{equation}
we deduce that, for all $N\in\nn$,
\begin{eqnarray}\label{tent-x}
\lf\|\ct\lf(f-\sum_{j=1}^N\lz_ja_j\r)\r\|_{\vlp}
&&\le\lf\|\sum_{j=N+1}^\fz|\lz_j|\ct(a_j)\r\|_{\vlp}\\
&&\ls\lf\|\lf\{\sum_{j= N+1}^\fz\lf[\frac{|\lz_j|\chi_{Q_j}}
{\|\chi_{Q_j}\|_{\vlp}}\r]^{\underline{p}}\r\}^{\frac1{\underline{p}}}
\r\|_{\vlp}.\noz
\end{eqnarray}
This, combined with \eqref{tent6} and the dominated convergence theorem
(see \cite[Theorem 2.62]{cfbook}), implies that
\begin{eqnarray*}
\lim_{N\to\fz}\lf\|\ct\lf(f-\sum_{j=1}^N\lz_ja_j\r)\r\|_{\vlp}
\ls\lf\|\lim_{N\to\fz}\lf\{\sum_{j= N+1}^\fz\lf[\frac{|\lz_j|\chi_{Q_j}}
{\|\chi_{Q_j}\|_{\vlp}}\r]^{\underline{p}}\r\}^{\frac1{\underline{p}}}
\r\|_{\vlp}=0.
\end{eqnarray*}
Therefore, \eqref{tent2} holds true in $T_2^{p(\cdot)}(\urn)$,
which completes the proof of Corollary \ref{c-tent1}.
\end{proof}

\begin{remark}\label{r-3.7x}
It was proved in \cite[Proposition 3.1]{jy10} that, if $f\in T_2^q(\urn)$ with
$q\in(0,\fz)$, then the decomposition \eqref{tent2} also holds true in
$T_2^q(\urn)$.
\end{remark}

Using Corollary \ref{c-tent1} and an argument
similar to that used in the proof of \eqref{tent-x}, we obtain the following atomic
characterization of $T_2^{p(\cdot)}(\urn)$, the details being omitted.

\begin{corollary}\label{c-tent}
Let $p(\cdot)\in C^{\log}(\rn)$.
Then $f\in T^{p(\cdot)}_2(\urn)$ if and only if there exist $\{\lz_j\}_{j\in\nn}\subset\cc$
and a sequence $\{a_j\}_{j\in\nn}$ of $(p(\cdot),\fz)$-atoms such that, for almost every
$(x,t)\in\urn$,
$
f(x,t)=\sum_{j\in\nn}\lz_ja_j(x,t)
$
and
\begin{equation*}
\int_\rn\lf\{\sum_{j\in\nn}
\lf[\frac{\lz_j\chi_{Q_j}}{\|\chi_{Q_j}\|_{\vlp}}\r]^{\underline{p}}
\r\}^{\frac{p(x)}{\underline{p}}}\,dx<\fz,
\end{equation*}
where, for each $j$, $Q_j$ denotes the cube appearing in the support of $a_j$;
moreover, for all $f\in T_2^{p(\cdot)}(\urn)$,
$
\|f\|_{T_2^{p(\cdot)}(\urn)}\sim\ca(\{\lz_j\}_{j\in\nn},\{Q_j\}_{j\in\nn})
$
with the implicit positive constants independent of $f$.
\end{corollary}

The following remark plays an important role in the proof of Theorem
\ref{t-dual}.

\begin{remark}\label{r-tent}
Let $p(\cdot)\in\cp(\rn)$ satisfy $p_+\in(0,1]$. Then, by \cite[Remark 4.4]{ns12},
we know that, for any $\{\lz_j\}_{j\in\nn}\subset\cc$
and cubes $\{Q_j\}_{j\in\nn}$ of $\rn$,
$$\sum_{j\in\nn}|\lz_j|\le
\ca(\{\lz_j\}_{j\in\nn},\{Q_j\}_{j\in\nn}).$$
\end{remark}

In what follows, let $T_{2,c}^{p(\cdot)}(\urn)$ and $T_{2,c}^{q}(\urn)$
with $q\in(0,\fz)$ be
the \emph{sets} of all functions, respectively, in $T_2^{p(\cdot)}(\urn)$
and $T_2^{q}(\urn)$ with compact supports.

\begin{proposition}\label{p-tent}
Let $p(\cdot)\in C^{\log}(\rn)$. Then
$T_{2,c}^{p(\cdot)}(\urn)\subset T_{2,c}^{2}(\urn)$ as sets.
\end{proposition}

\begin{proof}
By \cite[Lemma 3.3(i)]{jy10}, we know that, for any
$q\in(0,\fz)$, $T_{2,c}^q(\urn)\subset T_{2,c}^2(\urn)$.
Thus, to prove this proposition, it suffices to show
that $T_{2,c}^{p(\cdot)}(\urn)\subset T_{2,c}^{q_0}(\urn)$ for some $q_0\in(0,\fz)$.
To this end, suppose that $f\in T_{2,c}^{p(\cdot)}(\urn)$ and $\supp f\subset K$,
where $K$ is a compact set in $\urn$. Let $Q$ be a cube in $\rn$ such that
$K\subset \wh Q$. Then $\supp \ct(f)\subset Q$. From this and
the fact that $p_-\le p(x)$ for all $x\in\rn$, we deduce that
\begin{eqnarray*}
\int_\rn\lf[\ct(f)(x)\r]^{p_-}\,dx
&&\le \int_{\{x\in Q:\ \ct(f)(x)<1\}}\lf[\ct(f)(x)\r]^{p_-}\,dx
+ \int_{\{x\in Q:\ \ct(f)(x)\ge1\}}\cdots\\
&&\le |Q|+\int_\rn \lf[\ct(f)(x)\r]^{p(x)}\,dx<\fz,
\end{eqnarray*}
which implies that $T_{2,c}^{p(\cdot)}(\urn)\subset T_{2,c}^{p_-}(\urn)$
as sets and hence completes the proof of Proposition \ref{p-tent}.
\end{proof}

\subsection{Molecular characterizations of $H_L^{p(\cdot)}(\rn)$}
\hskip\parindent
In this subsection, we establish the molecular characterizations of $H_L^{p(\cdot)}(\rn)$.
We begin with some notions.
In what follows, for any $q\in(0,\fz)$, let $L^q(\urn)$ be the
\emph{set of all $q$-integrable functions on $\urn$}
and $L_{\loc}^q(\urn)$ the \emph{set of all locally $q$-integrable functions on $\urn$}.
For any $p(\cdot)\in\cp(\rn)$, let
\begin{equation}\label{3.6x}
s_0:=\lf\lfloor (n/m)(1/{p_-}-1)\r\rfloor,
\end{equation} namely,
$s_0$ denotes the largest integer smaller than or equal to $\frac nm(\frac1{p_-}-1)$.

Let $m$ be as in \eqref{assump1} and $s\in[s_0,\fz)$.
Let $C_{(m,s)}$ be a positive constant, depending on $m$ and $s$, such that
\begin{equation}\label{pi-def}
C_{(m,s)}\int_0^\fz t^{m(s+2)}e^{-2t^m}(1-e^{-t^m})^{s_0+1}\,\frac{dt}t=1.
\end{equation}
Let $q\in(0,\fz)$. Recall that the \emph{operator} $\pi_L$ is defined by setting, for all
$f\in T_{2,c}^q(\urn)$ and $x\in\rn$,
\begin{equation*}
\pi_L(f)(x):=C_{(m,s)}\int_0^\fz Q_{s,t^m}(I-P_{s_0,t^m})(f(\cdot,t))(x)\,\frac{dt}t.
\end{equation*}
Moreover, $\pi_L$ is well defined and $\pi_L(f)\in L^2(\rn)$ for all
$f\in T_{2,c}^q(\urn)$ (see \cite[p.\,4395]{yan08}).

\begin{remark}\label{r-3.14x}
Let $f\in L^2(\rn)$. Then, by \cite[(3.10)]{yan08}, we know that
$$f=C_{(m,s)}\int_0^\fz Q_{s,t^m}(I-P_{s_0,t^m})Q_{t^m}f\,\frac{dt}t,$$
where the integral converges in $L^2(\rn)$; see also \cite{adm96,mc86}.
\end{remark}

\begin{definition}\label{d-m}
Let $p(\cdot)\in C^{\log}(\rn)$ and $s\in[s_0,\fz)$ with $s_0$ as in \eqref{3.6x}. A measurable function
$\alpha$ on $\rn$ is called a \emph{$(p(\cdot),s,L)$-molecule} if there exists a
$(p(\cdot),\fz)$-atom $a$ supported on $Q$ for some cube $Q\st\rn$ such that,
 for all $x\in\rn$, $\alpha(x):=\pi_L(a)(x)$.

When it is necessary to specify the cube $Q$, then $a$ is called a
\emph{$(p(\cdot),s,L)$-molecule associated with $Q$}.
\end{definition}

\begin{remark}
Let $p(\cdot)\in C^{\log}(\rn)$ with $p_-\in(\frac n{n+\theta(L)},\fz)$,
where $p_-$ and $\theta(L)$ are as in
\eqref{2.1x} and \eqref{2.3x}, respectively.
Then, by Proposition \ref{p-tent2}(ii) below, we see that the $(p(\cdot),s,L)$-molecule
is well defined. Indeed, if $a$ is a $(p(\cdot),\fz)$-atom, by Corollary \ref{c-tent}, we then
know $a\in T_2^{p(\cdot)}(\urn)$, which, together with
Proposition \ref{p-tent2}(ii) below, implies that $\pi_L(a)\in H_L^{p(\cdot)}(\rn)$. Thus,
the $(p(\cdot),s,L)$-molecule is well defined.
\end{remark}

The molecular characterization of $H_L^{p(\cdot)}(\rn)$ is stated as follows.

\begin{theorem}\label{t-mol}
Let $p(\cdot)\in C^{\log}(\rn)$ satisfy $p_+\in(0,1]$
and $p_-\in(\frac n{n+\theta(L)},1]$, and $s\in[s_0,\fz)$
with $p_+$, $p_-$, $\theta(L)$ and $s_0$,
respectively, as in \eqref{2.1x}, \eqref{2.3x} and \eqref{3.6x}.

{\rm(i)} If $f\in H_L^{p(\cdot)}(\rn)$, then
there exist $\{\lz_j\}_{j\in\nn}\st\cc$ and a sequence $\{\alpha_j\}_{j\in\nn}$
of $(p(\cdot),s,L)$-molecules associated with cubes $\{Q_j\}_{j\in\nn}$
such that
$f=\sum_{j\in\nn}\lz_j \alpha_j$ in $H_L^{p(\cdot)}(\rn)$ and
$$\cb(\{\lz_j\alpha_j\}_{j\in\nn})
:=\ca(\{\lz_j\}_{j\in\nn},\{Q_j\}_{j\in\nn})\le C\|f\|_{H_L^{p(\cdot)}(\rn)}$$
with $C$ being a positive constant independent of $f$.

{\rm(ii)} Suppose that $\{\lz_k\}_{k\in\nn}\st\cc$ and
$\{\az_k\}_{k\in\nn}$ is a family of $(p(\cdot),s,L)$-molecules
satisfying $\cb(\{\lz_k\az_k\}_{k\in\nn})<\fz$.
Then $\sum_{k\in\nn}\lz_k\az_k$ converges in $H_L^{p(\cdot)}(\rn)$
and
$$\lf\|\sum_{k\in\nn}\lz_k\az_k\r\|_{H_L^{p(\cdot)}(\rn)}
\le C \cb(\{\lz_k\az_k\}_{k\in\nn})$$
with $C$ being a positive constant independent of $\{\lz_k\az_k\}_{k\in\nn}$.
\end{theorem}

The proof of Theorem \ref{t-dual} strongly depends on several
auxiliary estimates and will be presented later.
The following Lemma \ref{l-bigsball} is just \cite[Lemma 2.6]{zyl14}
(For the case when $p_->1$, see also \cite[Corollary 3.4]{Iz10}).

\begin{lemma}\label{l-bigsball}
Let $p(\cdot)\in C^{\log}(\rn)$.
Then there exists a positive
constant $C$ such that, for all cubes $Q_1$ and $Q_2$ satisfying $Q_1\st Q_2$,
\begin{equation*}
C^{-1}\lf(\frac{|Q_1|}{|Q_2|}\r)^{1/p_-}
\le\frac{\|\chi_{Q_1}\|_{\vlp}}{\|\chi_{Q_2}\|_{\vlp}}
\le C\lf(\frac{|Q_1|}{|Q_2|}\r)^{1/p_+},
\end{equation*}
where $p_-$ and $p_+$ are as in \eqref{2.1x}.
\end{lemma}
The following Fefferman-Stein vector-valued inequality
of the Hardy-Littlewood maximal operator $\cm$
on the space $\vlp$ was obtained in \cite[Corollary 2.1]{cfmp06}.

\begin{lemma}\label{l-hlmo}
Let $r\in(1,\fz)$ and $p(\cdot)\in C^{\log}(\rn)$.
If $p_-\in(1,\fz)$ with $p_-$ as in \eqref{2.1x}, then there exists a positive
constant $C$ such that, for all sequences $\{f_j\}_{j=1}^\fz$ of measurable
functions,
$$\lf\|\lf\{\sum_{j=1}^\fz \lf[\cm (f_j)\r]^r\r\}^{1/r}\r\|_{\vlp}
\le C\lf\|\lf(\sum_{j=1}^\fz|f_j|^r\r)^{1/r}\r\|_{\vlp}.$$
\end{lemma}

\begin{remark}\label{r-hlmo}
Let $k\in\nn$ and $p(\cdot)\in C^{\log}(\rn)$.
Then, by Lemma \ref{l-hlmo} and the fact that, for all cubes $Q\subset\rn$,
$r\in(0,p_-)$,
$\chi_{2^kQ}\le 2^{kn/r} [\cm(\chi_{Q})]^{1/r}$, we conclude that there exists
a positive constant $C$ such that,
for any $\{\lz_j\}_{j\in\nn}\subset\cc$
and cubes $\{Q_j\}_{j\in\nn}$ of $\rn$,
$$\ca(\{\lz_j\}_{j\in\nn},\{2^kQ_j\}_{j\in\nn})\le
C2^{kn(\frac1r-\frac1{p_+})}\ca(\{\lz_j\}_{j\in\nn},\{Q_j\}_{j\in\nn}),$$
where $p_-$ and $p_+$ are as in \eqref{2.1x}.
\end{remark}

\begin{proposition}\label{p-tent2}
Let $p(\cdot)\in C^{\log}(\rn)$ with $p_-\in(\frac n{n+\theta(L)},\fz)$
and $q\in(1,\fz)$, where $p_-$ and $\theta(L)$ are as in
\eqref{2.1x} and \eqref{2.3x}, respectively. Then

{\rm (i)} the operator $\pi_{L}$ is a bounded linear operator from
$T_2^q(\urn)$ to $L^q(\rn)$;

{\rm (ii)} the operator $\pi_{L}$, well defined on the space
$T_{2,c}^{p(\cdot)}(\urn)$, extends to a bounded linear operator from
$T_{2}^{p(\cdot)}(\urn)$ to $H_L^{p(\cdot)}(\rn)$.
\end{proposition}

\begin{proof}
To prove this proposition, it suffices to show (ii), since
(i) is just \cite[Lemma 3.4(a)]{yan08}.
Noticing that, due to Corollary \ref{c-tent}, $T_{2,c}^{p(\cdot)}(\urn)$ is a dense subset of
$T_2^{p(\cdot)}(\urn)$,
to prove (ii), we only need to
show that $\pi_L$ maps $T_{2,c}^{p(\cdot)}(\urn)$ continuously into $H_L^{p(\cdot)}$.

To this end, let $f\in T_{2,c}^{p(\cdot)}(\urn)$.
Then, by Proposition \ref{p-tent}, we know that
$f\in T_{2,c}^2(\urn)$ and hence $\pi_L$ is well defined on $T_{2,c}^{p(\cdot)}(\urn)$
by (i). This, combined with Lemma \ref{l-tent}, Corollary \ref{c-tent1} and
Remark \ref{r-3.7x},
implies that there exist sequences $\{\lz_j\}_{j\in\nn}\st\cc$
and $\{a_j\}_{j\in\nn}$ of $(p(\cdot),\fz)$-atoms such that,
for each $j\in\nn$, $\supp a_j\subset \wh Q_j$ with some cube $Q_j\st\rn$,
$f=\sum_{j\in\nn}\lz_ja_j$ in both $T_2^{p(\cdot)}(\urn)$
and $T_2^2(\urn)$, and
$\ca(\{\lz_j\}_{j\in\nn},\{Q_j\}_{j\in\nn})\ls\|f\|_{T_2^{p(\cdot)}(\urn)}$.
Thus, it follows from (i) that, for all $N\in\nn$,
\begin{equation*}
\lf\|\pi_L\lf(f-\sum_{j=1}^N\lz_ja_j\r)\r\|_{L^2(\rn)}
\ls\lf\|\sum_{j=N+1}^\fz\lz_ja_j\r\|_{T_2^2(\urn)}\to0
\end{equation*}
as $N\to\fz$; furthermore,
\begin{eqnarray}\label{0609-x}
\pi_L(f)&&=\lim_{N\to\fz}\sum_{j=1}^N\pi_L(\lz_ja_j)\\
&&=\sum_{j\in\nn}\lz_jC_{(m,s)}
\int_0^\fz Q_{s,t^m}(I-P_{s_0,t^m})(a_j(\cdot,t))\,\frac{dt}t
=:\sum_{j\in\nn}\lz_j\alpha_j\noz
\end{eqnarray}
in $L^2(\rn)$, where $s_0$ is as in \eqref{3.6x} and $s\in[s_0,\fz)$.

Next, we prove
$\lf\|S_{L}(\pi_L(f))\r\|_{\vlp}\ls\|f\|_{T_2^{p(\cdot)}(\urn)}$.
Observe that, for almost every $x\in\rn$,
\begin{equation*}
S_L(\pi_L(f))(x)
=S_L\lf(\sum_{j\in\nn}\lz_j\alpha_j\r)(x)
\le\sum_{j\in\nn}S_L(\lz_j\alpha_j)(x)
\end{equation*}
due to \eqref{0609-x}, the Fatou lemma and the fact that $S_L$ is bounded on $L^2(\rn)$
(see \eqref{SL-bounded}).
Then, by Remark \ref{r-vlp}(i) and the Fatou lemma of $\vlp$ (see \cite[Theorem 2.61]{cfbook}),
we see that
\begin{eqnarray}\label{pi-x}
\quad\lf\|S_{L}(\pi_L(f))\r\|_{\vlp}
\le\lf\{\sum_{i=0}^\fz\lf\|\lf(\sum_{j\in\nn}\lf[
S_L(\lz_j\alpha_j)\chi_{U_i(Q_j)}\r]^{\underline{p}}\r)^{\frac1{\underline{p}}}
\r\|_{\vlp}^{\underline{p}}\r\}^{\frac1{\underline{p}}},
\end{eqnarray}
where $\underline{p}$ is as in \eqref{3.1x} and, for any $j\in\nn$,
$U_0(Q_j):=4Q_j$ and $U_i(Q_j):=2^{i+2}Q_j\backslash (2^{i+1}Q_j)$
for all $i\in\nn$.

By \eqref{SL-bounded}, (i) and Lemma \ref{l-bigsball},
we find that, for all $q\in(1,\fz)$,
\begin{eqnarray*}
\lf\|S_L(\alpha_j)\r\|_{L^q(4Q_j)}
\ls \|\alpha_j\|_{L^q(\rn)}\sim\|\pi_L(a_j)\|_{L^q(\rn)}\ls \|a_j\|_{T_2^{q}(\urn)}
\ls|Q_j|^{\frac1q}\|\chi_{4Q_j}\|_{\vlp}^{-1}.
\end{eqnarray*}
From this, Lemma \ref{l-estimate} and Remark \ref{r-hlmo}, we deduce that
\begin{eqnarray}\label{pi-y}
&&\lf\|\lf(\sum_{j\in\nn}\lf[|\lz_j|S_L(\alpha_j)\chi_{4Q_j}
\r]^{\underline{p}}\r)^{\frac1{\underline{p}}}
\r\|_{\vlp}\\
&&\hs\ls\ca\lf(\{\lz_j\}_{j\in\nn},\{4Q_j\}_{j\in\nn}\r)
\ls\ca\lf(\{\lz_j\}_{j\in\nn},\{Q_j\}_{j\in\nn}\r)
\ls\|f\|_{T_2^{p(\cdot)}(\urn)}.\noz
\end{eqnarray}

Since $p_-\in (\frac n{n+\theta(L)},\fz)$,
we choose $\vez\in(0,\theta(L))$ such that $p_-\in (\frac n{n+\vez},\fz)$.
For $i\in\nn$, by \cite[(4.12)]{jyz09}, we know that, for all $x\in (4Q_j)^\complement$,
\begin{equation*}
S_L(\alpha_j)(x)\ls (r_{Q_j})^{\vez+\frac n2}|x-x_{Q_j}|^{-(n+\vez)}
\|a_j\|_{T_2^2(\urn)}.
\end{equation*}
Then, by this, the H\"older inequality and Lemma \ref{l-bigsball}, we further find that,
for any $q\in(p_+,\fz)\cap[2,\fz)$ with $p_+$ as in \eqref{2.1x},
\begin{eqnarray}\label{pi-y1}
\quad\|S_L(\alpha_j)\|_{L^q(U_i(Q_j))}
&&\ls2^{i(\frac nq-n-\vez)}|Q_j|^{\frac 1q-\frac12}\|a_j\|_{T_2^2(\urn)}
\ls 2^{i(\frac nq-n-\vez)}\|a_j\|_{T_2^q(\urn)}\\
&&\ls2^{-i(n+\vez)}|2^iQ_j|^{\frac1q}\|\chi_{Q_j}\|_{\vlp}^{-1}.\noz
\end{eqnarray}
Observe that, for all $r\in(\frac n{n+\vez},p_-)$,
$\chi_{2^iQ_j}\le 2^{\frac nri}[\cm(\chi_{Q_j})]^\frac1r$.
Thus, from this, \eqref{pi-y1} and Lemmas \ref{l-estimate} and \ref{l-hlmo}, we deduce that
\begin{eqnarray}\label{pi-z}
&&\lf\|\lf(\sum_{j\in\nn}\lf[|\lz_j|S_L(\alpha_j)\chi_{U_i(Q_j)}
\r]^{\underline{p}}\r)^{\frac1{\underline{p}}}
\r\|_{\vlp}\\
&&\hs\ls2^{-i(n+\vez)}\lf\|\lf\{\sum_{j\in\nn}
\lf[\frac{|\lz_j|\chi_{2^iQ_j}}
{\|\chi_{Q_j}\|_{\vlp}}\r]^{\underline{p}}\r\}^{\frac1{\underline{p}}}
\r\|_{\vlp}\noz\\
&&\hs\ls2^{-i(n+\vez-n/r)}\ca(\{\lz_j\}_{j\in\nn},\{Q_j\}_{j\in\nn})
\ls2^{-i(n+\vez-n/r)}\|f\|_{T_2^{p(\cdot)}(\urn)}.\noz
\end{eqnarray}

Combining \eqref{pi-x}, \eqref{pi-y} and \eqref{pi-z},
together with $r>\frac n{n+\vez}$, we conclude that
\begin{eqnarray*}
\lf\|S_{L}(\pi_L(f))\r\|_{\vlp}
&&\ls
\lf\{\sum_{i=0}^\fz2^{-i(n+\vez-\frac nr)\underline{p}}
\|f\|_{T_2^{p(\cdot)}(\urn)}\r\}^{1/\underline{p}}\ls\|f\|_{T_2^{p(\cdot)}(\urn)},
\end{eqnarray*}
which implies that $\pi_L$ is a bounded linear operator from
$T_{2}^{p(\cdot)}(\urn)$ to $H_L^{p(\cdot)}(\rn)$ and hence completes the proof of
Proposition \ref{p-tent2}.
\end{proof}

We now turn to the proof of Theorem \ref{t-mol}.

\begin{proof}[Proof of Theorem \ref{t-mol}]
We first prove (i).
Let $C_{(m,s)}$ be the constant as in \eqref{pi-def} and
$f\in H_L^{p(\cdot)}(\rn)\cap L^2(\rn)$.
Then, by Remark \ref{r-3.14x}, we see that
\begin{equation}\label{molecule-x1}
f=C_{(m,s)}\int_0^\fz Q_{s,t^m}(I-P_{s_0,t^m})Q_{t^m}f\,\frac{dt}t=\pi_L(Q_{t^m}f)
\end{equation}
in $L^2(\rn)$, where $s_0$ is as in \eqref{3.6x} and $s\in[s_0,\fz)$.
Since $f\in H_L^{p(\cdot)}(\rn)$, it follows that
$Q_{t^m}f\in T_2^{p(\cdot)}(\urn)$. Thus, by Lemma \ref{l-tent} and Corollary
\ref{c-tent1},
we find that $Q_{t^m}f=\sum_{j\in\nn}\lz_ja_j$ in the sense of both pointwise
and $T_2^{p(\cdot)}(\urn)$, where $\{\lz_j\}_{j\in\nn}\subset\cc$ and $\{a_j\}_{j\in\nn}$ are
$(p(\cdot),\fz)$-atoms satisfying that, for each $j\in\nn$,
$\supp a_j\subset Q_j$ with some cube $Q_j\subset\rn$, and
$$\ca\lf(\{\lz_j\}_{j\in\nn},\{Q_j\}_{j\in\nn}\r)
\ls\|Q_{t^m}f\|_{T_2^{p(\cdot)}(\urn)}\sim\|f\|_{H_L^{p(\cdot)}(\rn)}.$$
For any $j\in\nn$, let $\alpha_j:=\pi_L(a_j)$.
Then $\alpha_j$ is a $(p(\cdot),s,L)$-molecule and,
by \eqref{molecule-x1} and Proposition \ref{p-tent2}, we conclude that
\begin{equation*}
f=\pi_L(Q_{t^m}f)=\sum_{j\in\nn}\lz_j\pi_L(a_j)
=:\sum_{j\in\nn}\lz_j\alpha_j
\end{equation*}
in both $L^2(\rn)$ and $H_L^{p(\cdot)}(\rn)$.

Now, for any $f\in H_L^{p(\cdot)}(\rn)$, since $H_L^{p(\cdot)}(\rn)\cap L^2(\rn)$
is dense in $H_L^{p(\cdot)}(\rn)$, it follows that there exists a sequence
$\{f_k\}_{k\in\nn}\subset [H_L^{p(\cdot)}(\rn)\cap L^2(\rn)]$ such that,
for all $k\in\nn$,
\begin{equation*}
\|f-f_k\|_{H_L^{p(\cdot)}(\rn)}\le 2^{-k}\|f\|_{H_L^{p(\cdot)}(\rn)}.
\end{equation*}
Let $f_0:=0$. Then
\begin{equation}\label{molecule-z1}
f=\sum_{k\in\nn}(f_k-f_{k-1})\ {\rm in}\ H_L^{p(\cdot)}(\rn).
\end{equation}
From the above argument, we deduce that, for each $k\in\nn$, there exist
$\{\lz_j^k\}_{k\in\nn}\st\cc$ and a sequence $\{\az_j^k\}_{j\in\nn}$ of
$(p(\cdot),s,L)$-molecules such that
\begin{equation}\label{molecule-z2}
f_k-f_{k-1}=\sum_{j\in\nn}\lz_j^k\alpha_j^k\
{\rm in}\ H_L^{p(\cdot)}(\rn)
\end{equation}
and
$$\ca(\{\lz_j^k\}_{j\in\nn},\{Q_j^k\}_{j\in\nn})
\ls \|f_k-f_{k-1}\|_{H_L^{p(\cdot)}(\rn)}
\ls2^{-k}\|f\|_{H_L^{p(\cdot)}(\rn)},$$
where, for any $k,\,j\in\nn$, $Q_j^k$ denotes the cube appearing in the definition
of the $(p(\cdot),s,L)$-molecule $\az_j^k$.
By this, the Minkowski inequality, (ii) and (iii) of Remark \ref{r-vlp}
and \eqref{simple-ineq}, we see that
\begin{eqnarray*}
&&\int_\rn\lf\{\sum_{k\in\nn}\sum_{j\in\nn}\lf[\frac{|\lz_j^k|\chi_{Q_j^k}}
{\|f\|_{H_L^{p(\cdot)}(\rn)}\|\chi_{Q_j^k}\|_{\vlp}}
\r]^{\underline{p}}\r\}^{\frac{p(x)}{\underline{p}}}\,dx\\
&&\hs\le\int_\rn\lf\{\sum_{k\in\nn}\lf[\sum_{j\in\nn}
\lf(\frac{|\lz_j^k|\chi_{Q_j^k}}{\|f\|_{H_L^{p(\cdot)}(\rn)}\|\chi_{Q_j^k}\|_{\vlp}}
\r)^{\underline{p}}\r]^{p(x)}\r\}^{\frac1{\underline{p}}}\,dx\\
&&\hs\ls\lf(\sum_{k\in\nn}\lf\{\int_\rn2^{-kp(x)}\lf[\sum_{j\in\nn}
\lf(\frac{|\lz_j^k|\chi_{Q_j^k}}{2^{-k}\|f\|_{H_L^{p(\cdot)}(\rn)}
\|\chi_{Q_j^k}\|_{\vlp}}\r)^{\underline{p}}\r]^{\frac{p(x)}{\underline{p}}}\,dx\r\}^{\underline{p}}\r)^{\frac1{\underline{p}}}\\
&&\hs\ls\lf(\sum_{k\in\nn}2^{-k\underline{p}^2}\r)^{1/\underline{p}}<\fz,
\end{eqnarray*}
which implies that
$\cb\lf(\{\lz_j^k\alpha_j^k\}_{j,k\in\nn}\r)\ls\|f\|_{H_L^{p(\cdot)}(\rn)}$.
Moreover, by \eqref{molecule-z1} and \eqref{molecule-z2}, we conclude that
$$f=\sum_{k\in\nn}\sum_{j\in\nn}\lz_j^k\alpha_j^k$$
in $H_L^{p(\cdot)}(\rn)$ and hence the proof of (i) is completed.

Next, we show (ii). Without loss of generality, we may assume that, for each
$k\in\nn$, $\az_k:=\pi_L(a_k)$, where $a_k$ is a $(p(\cdot),\fz)$-atom
supported on $\wh Q_k$ for some cube $Q_k\st\rn$.
Then, from Proposition \ref{p-tent2}(ii) and Corollary \ref{c-tent},
we deduce that, for all $N_1,\,N_2\in\nn$ with $N_1<N_2$,
\begin{eqnarray*}
\lf\|\sum_{k=N_1}^{N_2}\lz_k\az_k\r\|_{H_L^{p(\cdot)}(\rn)}
&&\sim \lf\|\pi_L\lf(\sum_{k=N_1}^{N_2}\lz_ka_k\r)\r\|_{H_L^{p(\cdot)}(\rn)}
\ls\lf\|\sum_{k=N_1}^{N_2}\lz_ka_k\r\|_{T_{2}^{p(\cdot)}(\urn)}\\
&&\ls \lf\|\lf\{\sum_{k=N_1}^{N_2}\lf[\frac{|\lz_k|\chi_{Q_k}}{\|\chi_{Q_k}\|_{\vlp}}
\r]^{\underline{p}}\r\}^{\frac1{\underline{p}}}\r\|_{\vlp},
\end{eqnarray*}
which tends to zero as $N_1,\,N_2\to\fz$ due to the dominated convergence
theorem. Thus, $\sum_{k\in\nn}\lz_k\az_k$
converges in $H_L^{p(\cdot)}(\rn)$ and, by the Fatou lemma of $\vlp$,
Proposition \ref{p-tent2}(ii) and Corollary \ref{c-tent}, we further know that
\begin{eqnarray*}
&&\lf\|\sum_{k=1}^{\fz}\lz_k\az_k\r\|_{H_L^{p(\cdot)}(\rn)}\\
&&\hs\le \liminf_{N\to\fz}\lf\|\sum_{k=1}^{N}\lz_k\az_k\r\|_{H_L^{p(\cdot)}(\rn)}
=\liminf_{N\to\fz}\lf\|\pi_L\lf(\sum_{k=1}^{N}\lz_ka_k\r)\r\|_{H_L^{p(\cdot)}(\rn)}\\
&&\hs\ls \liminf_{N\to\fz}\lf\|\sum_{k=1}^{N}\lz_ka_k\r\|_{T_2^{p(\cdot)}(\urn)}
\ls\liminf_{N\to\fz}\lf\|\lf\{\sum_{k=1}^{N}
\lf[\frac{|\lz_k|\chi_{Q_k}}{\|\chi_{Q_k}\|_{\vlp}}
\r]^{\underline{p}}\r\}^{\frac1{\underline{p}}}\r\|_{\vlp}\\
&&\hs\ls \cb(\{\lz_k\az_k\}_{k\in\nn}),
\end{eqnarray*}
which completes the proof of Theorem \ref{t-mol}.
\end{proof}

In what follows, for all $s\in[s_0,\fz)$ with $s_0$ as in \eqref{3.6x} and $p(\cdot)\in\cp(\rn)$,
denote by $H_{L,{\rm fin}}^{p(\cdot)}(\rn)$ the \emph{set of finite linear combinations of
$(p(\cdot),s,L)$-molecules}.
For any $f\in H_{L,{\rm fin}}^{p(\cdot)}(\rn)$, the quasi-norm is given by
$$\|f\|_{H_{L,{\rm fin}}^{p(\cdot)}(\rn)}
:=\inf\lf\{\cb(\{\lz_j\alpha_j\}_{j=1}^N):\ N\in \nn,\
f=\sum_{j=1}^N\lz_j\alpha_j\r\},$$
where the infimum is taken over all finite
molecular decompositions of $f$.

\begin{corollary}\label{c-dense}
Let $p(\cdot)\in C^{\log}(\rn)$ satisfy
$p_+\in(0,1]$ and $p_-\in(\frac n{n+\theta(L)},1]$,
where $p_-$, $p_+$ and $\theta(L)$ are,
respectively, as in \eqref{2.1x} and \eqref{2.3x}.
Then $H_{L,{\rm fin}}^{p(\cdot)}(\rn)$ is dense in $H_L^{p(\cdot)}(\rn)$.
\end{corollary}

\begin{proof}
Let $f\in H_L^{p(\cdot)}(\rn)\cap L^2(\rn)$. Then
$Q_{t^m}f\in T_2^{p(\cdot)}(\urn)$ and hence, by Lemma \ref{l-tent},
we have $Q_{t^m}f=\sum_k\lz_ka_k$ in $T_2^{p(\cdot)}(\urn)$,
where $\{\lz_k\}_{k\in\nn}\st\cc$ and $\{a_k\}_{k\in\nn}$ are $(p(\cdot),\fz)$-atoms.
For every $k\in\nn$, let $\alpha_k:=\pi_L(a_k)$. Then $\{\az_k\}_{k\in\nn}$
are $(p(\cdot),s,L)$-molecules. Thus,
by Proposition \ref{p-tent2}(ii), we conclude that, for all $N\in\nn$,
\begin{eqnarray*}
\lf\|f-\sum_{k=1}^N\lz_k\az_k\r\|_{H_L^{p(\cdot)}(\rn)}
&&=\lf\|\pi_L\lf(Q_{t^m}f-\sum_{k=1}^N\lz_ka_k\r)\r\|_{H_L^{p(\cdot)}(\rn)}\\
&&\ls\lf\|Q_{t^m}f-\sum_{k=1}^N\lz_ka_k\r\|_{T_2^{p(\cdot)}(\urn)}
\to0,
\end{eqnarray*}
as $N\to\fz$, which implies that $H_{L,{\rm fin}}^{p(\cdot)}(\rn)$ is dense in
$H_L^{p(\cdot)}(\rn)\cap L^2(\rn)$ and hence in $H_L^{p(\cdot)}(\rn)$.
This finishes the proof of Corollary \ref{c-dense}.
\end{proof}

\section{BMO-type spaces and the Duality of $H_L^{p(\cdot)}(\rn)$\label{s4}}
\hskip\parindent
In this section, we mainly consider the duality of $H_L^{p(\cdot)}(\rn)$.
To this end, motivated by \cite{jyz09},
we introduce the space ${\rm BMO}_{p(\cdot),L}^{s}(\rn)$
associated with the operator $L$ and the variable exponent $p(\cdot)$.

\begin{definition}\label{d-bmo}
Let $L$ satisfy Assumptions \ref{as-a} and \ref{as-b}, $p(\cdot)\in C^{\log}(\rn)$
with $p_+\in(0,1]$ and $s\in[s_0,\fz)$,
where $p_+$ and $s_0$ are, respectively, as in \eqref{2.1x} and \eqref{3.6x}.
Then the \emph{{\rm BMO}-type space}
${\rm BMO}_{p(\cdot),L}^{s}(\rn)$ is defined to be the set of all
functions $f\in\cm(\rn)$ such that $\|f\|_{{\rm BMO}_{p(\cdot),L}^{s}(\rn)}<\fz$,
where
$$\|f\|_{{\rm BMO}_{p(\cdot),L}^{s}(\rn)}
:=\sup_{Q\subset\rn}\frac{|Q|^{1/2}}{\|\chi_Q\|_{\vlp}}\lf\{\int_Q
|f(x)-P_{s,(r_Q)^m}f(x)|^2\,dx\r\}^{\frac12}$$
and the supremum is taken over all cubes $Q$ of $\rn$.
\end{definition}

\begin{remark}
(i) The space $({\rm BMO}_{p(\cdot),L}^{s}(\rn),
\|\cdot\|_{{\rm BMO}_{p(\cdot),L}^{s}(\rn)})$ is a vector space
with the semi-norm vanishing on the space $\mathcal K_{(L,s)}(\rn)$
which is defined by
\begin{eqnarray*}
\mathcal K_{(L,s)}(\rn):=&&\lf\{f\in\cm(\rn):\ P_{s,t}f(x)=f(x)\r.\\
&&\quad\quad\quad\lf.{\rm for\ almost\ every}\ x\in\rn\ {\rm and\ all}\ t\in(0,\fz)
\r\}.
\end{eqnarray*}
In this article, the space ${\rm BMO}_{p(\cdot),L}^{s}(\rn)$ is understood
to be modulo $\mathcal K_{(L,s)}(\rn)$; see \cite[Section 6]{dy05} for a discussion
of $\mathcal K_{(L,0)}(\rn)$ when $L$ is a second
order elliptic operator of divergence form or a Schr\"odinger operator.

(ii) If $p(\cdot)\equiv1$ and $s=0$,
then ${\rm BMO}_{p(\cdot),L}^{s}(\rn)$ is just ${\rm BMO}_L(\rn)$ introduced by Duong
and Yan \cite{dy05}. If $p(\cdot)\in\cp(\rn)$ is defined by
$\frac1{p(\cdot)}:=\alpha+\frac12$ for some constant $\alpha\in(0,\frac{\theta(L)}n)$,
then ${\rm BMO}_{p(\cdot),L}^{s}(\rn)$ becomes the space
$\mathfrak{L}_L(\alpha,2,s)$ studied in \cite{yan08}.
\end{remark}

Now we state the main result of this section as follows.

\begin{theorem}\label{t-dual}
Let $p(\cdot)\in C^{\log}(\rn)$ satisfy $p_+\in(0,1]$
and $p_-\in(\frac n{n+\theta(L)},1]$
with $p_+$, $p_-$ and $\theta(L)$,
respectively, as in \eqref{2.1x} and \eqref{2.3x}.
Let $s_0$ be as in \eqref{3.6x} and $L^\ast$ denote the adjoint operator of $L$.
Then $(H_L^{p(\cdot)}(\rn))^\ast$ coincides with ${\rm BMO}_{p(\cdot),L^\ast}^{s_0}(\rn)$
in the following sense:

{\rm(i)} If $g\in {\rm BMO}_{p(\cdot),L^\ast}^{s_0}(\rn)$, then the linear
mapping $\ell$, which is initially defined on $H_{L,{\rm fin}}^{p(\cdot)}(\rn)$ by
\begin{equation}\label{dual-y}
\ell_g(f):=\int_\rn f(x)g(x)\,dx,
\end{equation}
extends to a bounded linear functional on $H_L^{p(\cdot)}(\rn)$ and
$$\|\ell_g\|_{(H_L^{p(\cdot)}(\rn))^\ast}\le C\|g\|_{{\rm BMO}_{p(\cdot),L^\ast}^{s_0}(\rn)},$$
where $C$ is a positive constant independent of $g$.

{\rm(ii)} Conversely, let $\ell$ be a bounded linear functional on
$H_L^{p(\cdot)}(\rn)$. Then $\ell$ has the form as in \eqref{dual-y} with
a unique $g\in {\rm BMO}_{p(\cdot),L^\ast}^{s_0}(\rn)$ for all
$f\in H_{L,{\rm fin}}^{p(\cdot)}(\rn)$ and
$$\|g\|_{{\rm BMO}_{p(\cdot),L^\ast}^{s_0}(\rn)}
\le \wz C\|\ell\|_{(H_L^{p(\cdot)}(\rn))^\ast},$$
where $\wz C$ is a positive constant independent of $\ell$.
\end{theorem}

\begin{remark}\label{r-dual} Let $p(\cdot)$, $s_0$, $L$, $L^*$ and $\theta(L)$ be as in Theorem \ref{t-dual}.

(i) If $f\in H_L^{p(\cdot)}(\rn)$, then, from Theorem \ref{t-dual},
we deduce that $f$ satisfies the cancelation condition
$\int_\rn f(x)g(x)\,dx=0$ for all $g\in \mathcal K_{L^\ast,s_0}(\rn)$, since,
if $g\in \mathcal K_{L^\ast,s_0}(\rn)$, then $\|g\|_{{\rm BMO}_{p(\cdot),L^\ast}^{s_0}(\rn)}=0$.
Observe that, if $g\in \mathcal K_{L^\ast,s_0}(\rn)$, then $g$ is not necessary to be zero
almost everywhere and hence, if $f$ is a smooth function with compact support, then
$\int_\rn f(x)g(x)\,dx$ may not equal zero. Therefore, smooth functions with
compact supports are not necessary
to be in $H_L^{p(\cdot)}(\rn)$; see also \cite[p.\,962]{dy05}.

(ii) Observe that, by the proof of Corollary \ref{c-dense}, we see that
$$H_{L,{\rm fin}}^{p(\cdot)}(\rn)\st [H_L^{p(\cdot)}(\rn)\cap L^2(\rn)]$$
and $H_{L,{\rm fin}}^{p(\cdot)}(\rn)$ is dense in $H_L^{p(\cdot)}(\rn)\cap L^2(\rn)$.
From this, it follows that, if we require that \eqref{dual-y} holds true for
all $f\in H_L^{p(\cdot)}(\rn)\cap L^2(\rn)$ instead
of all $f\in H_{L,{\rm fin}}^{p(\cdot)}(\rn)$, then all conclusions of Theorem \ref{t-dual}
also hold true.
\end{remark}

To prove Theorem \ref{t-dual}, we need some preparations.

\begin{proposition}\label{p-bmo1}
Let $p(\cdot)\in C^{\log}(\rn)$ satisfy $p_+\in(0,1]$ and
$p_-\in(\frac n{n+\theta(L)},1]$,
and $s\in[s_0,\fz)$,
where $p_+$, $p_-$, $\theta(L)$ and $s_0$ are, respectively, as in
\eqref{2.1x}, \eqref{2.3x} and \eqref{3.6x}.
Then there exists
a positive constant $C$ such that, for all $t\in(0,\fz)$, $K\in(1,\fz)$,
$f\in {\rm BMO}_{p(\cdot),L}^{s}(\rn)$ and $x\in\rn$, when $p_+=1$,
\begin{eqnarray}\label{bmo1x}
&&\lf|P_{s,t}f(x)-P_{s,Kt}f(x)\r|\\
&&\hs\le C(1+\log_2K)\lf\|\chi_{Q(x,(Kt)^{\frac1m})}\r\|_{\vlp}
|Q(x,(Kt)^{\frac1m})|^{-1}\|f\|_{{\rm BMO}_{p(\cdot),L}^{s}(\rn)}\noz
\end{eqnarray}
and, when $p_+\in(0,1)$,
\begin{eqnarray}\label{bmo1x1}
&&\lf|P_{s,t}f(x)-P_{s,Kt}f(x)\r|\\
&&\hs\le C\lf\|\chi_{Q(x,(Kt)^{\frac1m})}\r\|_{\vlp}
|Q(x,(Kt)^{\frac1m})|^{-1}\|f\|_{{\rm BMO}_{p(\cdot),L}^{s}(\rn)}.\noz
\end{eqnarray}
\end{proposition}
\begin{proof}
Without loss of generality, we may assume that $\|f\|_{{\rm BMO}_{p(\cdot),L}^{s}(\rn)}=1$.
We claim that, for all $t,\ v\in(0,\fz)$ with
$\frac t2\le v\le 2t$, and $x\in\rn$,
\begin{equation}\label{bmo1}
\lf|P_{s,t}f(x)-P_{s,v}f(x)\r|\ls \lf\|\chi_{Q(x,t^\frac1m)}\r\|_{\vlp}|Q(x,t^{\frac1m})|^{-1}.
\end{equation}
If this claim holds true, then, by Lemma \ref{l-bigsball}, we see that
\begin{eqnarray*}
\lf|P_{s,t}f(x)-P_{s,Kt}f(x)\r|
&&\le \sum_{i=0}^{l-1}\lf|P_{s,2^it}f(x)-P_{s,2^{i+1}t}f(x)\r|
+\lf|P_{s,2^lt}f(x)-P_{s,Kt}f(x)\r|\\
&&\ls\sum_{i=0}^{l-1}\frac{\|\chi_{Q(x,(2^it)^\frac1m)}\|_{\vlp}}
{|Q(x,(2^it)^{\frac1m})|}
+\frac{\|\chi_{Q(x,(Kt)^\frac1m)}\|_{\vlp}}{|Q(x,(Kt)^{\frac1m})|}\\
&&\ls\lf\{\sum_{i=0}^{l-1}\lf[\frac{|Q(x,(2^it)^{\frac1m})|}
{|Q(x,(Kt)^{\frac1m})|}\r]^{\frac1{p_+}-1}
+1\r\}\frac{\|\chi_{Q(x,(Kt)^\frac1m)}\|_{\vlp}}{|Q(x,(Kt)^{\frac1m})|},
\end{eqnarray*}
where $l:=\lfloor \log_2K\rfloor$.
By this, we further conclude that, when $p_+=1$,
$$\lf|P_{s,t}f(x)-P_{s,Kt}f(x)\r|
\ls(1+\log_2K)\lf\|\chi_{Q(x,(Kt)^\frac1m)}\r\|_{\vlp}|Q(x,(Kt)^{\frac1m})|^{-1}$$
and, when $p_+\in(0,1)$,
$$\lf|P_{s,t}f(x)-P_{s,Kt}f(x)\r|
\ls\lf\|\chi_{Q(x,(Kt)^\frac1m)}\r\|_{\vlp}|Q(x,(Kt)^{\frac1m})|^{-1},$$
which implies that \eqref{bmo1x} and \eqref{bmo1x1} hold true.

Therefore, to complete the proof of this proposition,
it remains to prove the above claim.
By the commutative properties of semigroups, we have
\begin{equation}\label{bmo2x}
P_{s,t}f-P_{s,v}f=P_{s,t}(f-P_{s,v}f)-P_{s,v}(f-P_{s,t}f).
\end{equation}
Since $\theta(L)\in(n[\frac1{p_-}-1],\fz)$ and $p_+\in(0,1]$,
it follows that there exists $\vez\in(0,\theta(L))$ such that
\begin{equation}\label{bmo1y}
\vez>n\lf(\frac1{p_-}-1\r)>n\lf(\frac1{p_-}-\frac1{p_+}\r).
\end{equation}
From \eqref{operator-k}, Assumption \ref{as-a}, the H\"older inequality,
Lemma \ref{l-bigsball} and the fact that $\frac t2\le v\le 2t$,
we deduce that, for all $x\in\rn$,
\begin{eqnarray}\label{bmo2}
\quad\lf|P_{s,t}(f-P_{s,v}f)(x)\r|
&&\ls t^{-\frac nm}\int_\rn g\lf(\frac{|x-y|}{t^{1/m}}\r)
|f(y)-P_{s,v}f(y)|\,dy\\
&&\ls\lf\{v^{-\frac nm}\int_{Q(x,v^{\frac1m})}|f(y)-P_{s,v}f(y)|^2\,dy\r\}^{1/2}\noz\\
&&\quad\quad+v^{-\frac nm}\sum_{i=1}^\fz
\int_{S_i}g\lf(\frac{|x-y|}{t^{\frac1m}}\r)|f(y)-P_{s,v}f(y)|\,dy\noz\\
&&\ls\lf\|\chi_{Q(x,t^\frac1m)}\r\|_{\vlp}|Q(x,t^\frac1m)|^{-1}
\|f\|_{{\rm BMO}_{p(\cdot),L}^{s}(\rn)}\noz\\
&&\quad\quad+\sum_{i=1}^\fz2^{-\frac{n+\vez}mi}v^{-\frac nm}\int_{Q(x,(2^iv)^{\frac1m})}
|f(y)-P_{s,v}f(y)|\,dy,\noz
\end{eqnarray}
where, for each $i\in\nn$, $S_i:=Q(x,(2^iv)^{\frac1m})\backslash
Q(x,(2^{i-1}v)^{\frac1m})$.
Notice that, for any $i\in\nn$, there exists a collection
$\{Q_{i,j}\}_{j=1}^{N_i}$ of cubes with $N_i\sim 2^{ni/m}$ such that $\ell(Q_{i,j})=v^{1/m}$
and $Q(x,(2^iv)^{1/m})\subset\bigcup_{j=1}^{N_i}Q_{i,j}$.
Thus, by the H\"older inequality and Lemma \ref{l-bigsball}, we find that
\begin{eqnarray*}
\quad\int_{Q(x,(2^iv)^{\frac1m})}|f(y)-P_{s,v}f(y)|\,dy
&&\le\sum_{j=1}^{N_i}\int_{Q_{i,j}}|f(y)-P_{s,v}f(y)|\,dy\\
&&\ls\sum_{j=1}^{N_i}\lf\{\int_{Q_{i,j}}
|f(y)-P_{s,v}f(y)|^2\,dy\r\}^{\frac12}|Q_{i,j}|^{\frac12}\noz\\
&&\ls\|f\|_{{\rm BMO}_{p(\cdot),L}^{s}(\rn)}\sum_{j=1}^{N_i}
\|\chi_{Q_{i,j}}\|_{\vlp}\noz\\
&&\ls 2^{\frac nmi}2^{-\frac{n}{m}(\frac1{p_+}-\frac1{p_-})i}
\lf\|\chi_{Q(x,v^{\frac1m})}\r\|_{\vlp},\noz
\end{eqnarray*}
which, together with Lemma \ref{l-bigsball} again and \eqref{bmo1y},
implies that
\begin{eqnarray*}
&&\sum_{i=1}^\fz2^{-i\frac{n+\vez}m}v^{-\frac nm}\int_{Q(x,(2^iv)^{\frac1m})}
|f(y)-P_{s,v}f(y)|\,dy\\
&&\hs\ls v^{-\frac nm}\sum_{i=1}^\fz2^{-\frac nm[\frac\vez n+\frac1{p_+}-\frac1{p_-}]i}
\lf\|\chi_{Q(x,t^{\frac1m})}\r\|_{\vlp}
\sim\lf\|\chi_{Q(x,t^{\frac1m})}\r\|_{\vlp}|Q(x,t^{\frac1m})|^{-1}.
\end{eqnarray*}
By this and \eqref{bmo2}, we further conclude that, for all $x\in\rn$,
\begin{equation}\label{bmo3}
|P_{s,t}(f-P_{s,v}f)(x)|
\ls\lf\|\chi_{Q(x,t^{\frac1m})}\r\|_{\vlp}|Q(x,t^{\frac1m})|^{-1}.
\end{equation}

By an argument similar to that used in the proof of \eqref{bmo3},
we also see that, for all $x\in\rn$,
\begin{equation*}
|P_{s,v}(f-P_{s,t}f)(x)|\ls \lf\|\chi_{Q(x,t^{\frac1m})}\r\|_{\vlp}|Q(x,t^{\frac1m})|^{-1},
\end{equation*}
which, combined with \eqref{bmo2x} and \eqref{bmo3}, implies that \eqref{bmo1} holds
true. This finishes the proof of Proposition \ref{p-bmo1}.
\end{proof}

\begin{proposition}\label{p-bmo2}
Let $p(\cdot)$ and $s$ be as in Proposition \ref{p-bmo1}. Then, for
any $\delta\in(n[\frac 1{p_-}-1],\fz)$, there exists a positive constant $C$ such that,
for all $f\in {\rm BMO}_{p(\cdot),L}^{s}(\rn)$, $t\in(0,\fz)$ and $x\in\rn$,
\begin{equation*}
\int_{\rn}\frac{|f(y)-P_{s,t}f(y)|}{(t^\frac1m+|x-y|)^{n+\delta}}\,dy
\le C t^{-\frac{n+\delta}{m}}\lf\|\chi_{Q(x,t^{\frac1m})}\r\|_{\vlp}
\|f\|_{{\rm BMO}_{p(\cdot),L}^{s}(\rn)}.
\end{equation*}
\end{proposition}
\begin{proof}
For all $t\in(0,\fz)$ and $x\in\rn$, we write
$$\int_{\rn}\frac{|f(y)-P_{s,t}f(y)|}{(t^\frac1m+|x-y|)^{n+\delta}}\,dy
=\int_{Q(x,t^{\frac1m})}\frac{|f(y)-P_{s,t}f(y)|}{(t^\frac1m+|x-y|)^{n+\delta}}\,dy
+\int_{[Q(x,t^{\frac1m})]^\complement}\cdots=:{\rm I}_1+{\rm I}_2.$$
Obviously, by the H\"older inequality, we easily see that
$${\rm I}_1\ls t^{-\frac{n+\delta}{m}}\int_{Q(x,t^\frac1m)}|f(y)-P_{s,t}f(y)|\,dy
\ls t^{-\frac{n+\delta}m}\lf\|\chi_{Q(x,t^{\frac1m})}\r\|_{\vlp}
\|f\|_{{\rm BMO}_{p(\cdot),L}^{s}(\rn)}.$$

To estimate ${\rm I}_2$, we first notice that
\begin{eqnarray}\label{bmo4x}
{\rm I}_2
&&\ls\sum_{k=1}^\fz(2^kt)^{-\frac{n+\delta}{m}}
\int_{Q(x,(2^kt)^{\frac 1m})}|f(y)-P_{s,t}f(y)|\,dy\\
&&\ls\sum_{k=1}^\fz(2^kt)^{-\frac{n+\delta}{m}}
\int_{Q(x,(2^kt)^{\frac 1m})}|f(y)-P_{s,2^kt}f(y)|\,dy\noz\\
&&\quad\quad+\sum_{k=1}^\fz(2^kt)^{-\frac{n+\delta} m}\sup_{y\in Q(x,(2^kt)^{\frac1m})}
|P_{s,t}f(y)-P_{s,2^kt}f(y)||Q(x,(2^kt)^{\frac1m})|\noz\\
&&=:{\rm I}_{2,1}+{\rm I}_{2,2}.\noz
\end{eqnarray}
By the H\"older inequality, Lemma \ref{l-bigsball}
and the fact that $\delta>n(\frac 1{p_-}-1)$, we find that
\begin{eqnarray}\label{bmo4y}
{\rm I}_{2,1}
&&\ls\sum_{k=1}^\fz(2^kt)^{-\frac{n+\delta}m}
\lf\|\chi_{Q(x,(2^kt)^{\frac1m})}\r\|_{\vlp}
\|f\|_{{\rm BMO}_{p(\cdot),L}^{s}(\rn)}\\
&&\ls t^{-\frac{n+\delta}m}\lf\|\chi_{Q(x,t^{\frac1m})}\r\|_{\vlp}
\|f\|_{{\rm BMO}_{p(\cdot),L}^{s}(\rn)}\sum_{k=1}^\fz 2^{-k(n+\delta-\frac n{p_-})}\noz\\
&&\ls t^{-\frac{n+\delta}m}\lf\|\chi_{Q(x,t^{\frac1m})}\r\|_{\vlp}
\|f\|_{{\rm BMO}_{p(\cdot),L}^{s}(\rn)}.\noz
\end{eqnarray}
For ${\rm I}_{2,2}$, by Proposition \ref{p-bmo1} and Lemma \ref{l-bigsball},
we know that
\begin{eqnarray*}
{\rm I}_{2,2}
&&\ls \sum_{k=1}^\fz k(2^kt)^{-\frac{n+\delta}m}
\sup_{y\in Q(x,(2^kt)^{\frac1m})}\lf\|\chi_{Q(y,(2^kt)^{\frac1m})}\r\|_{\vlp}
\|f\|_{{\rm BMO}_{p(\cdot),L}^{s}(\rn)}\\
&&\ls \sum_{k=1}^\fz k(2^kt)^{-\frac{n+\delta}m}
\lf\|\chi_{Q(x,(2^kt)^{\frac1m})}\r\|_{\vlp}
\|f\|_{{\rm BMO}_{p(\cdot),L}^{s}(\rn)}\\
&&\ls t^{-\frac{n+\delta}m}\lf\|\chi_{Q(x,t^{\frac1m})}\r\|_{\vlp}
\|f\|_{{\rm BMO}_{p(\cdot),L}^{s}(\rn)}.
\end{eqnarray*}
This, together with \eqref{bmo4x} and \eqref{bmo4y}, implies that
\begin{eqnarray*}
{\rm I}_2\ls
t^{-\frac{n+\delta} m}\lf\|\chi_{Q(x,t^{\frac1m})}\r\|_{\vlp}
\|f\|_{{\rm BMO}_{p(\cdot),L}^{s}(\rn)},
\end{eqnarray*}
which completes the proof of Proposition \ref{p-bmo2}.
\end{proof}

Let $p(\cdot)\in\cp(\rn)$. Recall that a measure $d\mu$ on $\urn$ is
called a \emph{$p(\cdot)$-Carleson measure} if
\begin{equation*}
\|d\mu\|_{p(\cdot)}:=\sup_{Q\subset\rn}
\frac{|Q|^{1/2}}{\|\chi_Q\|_{\vlp}}\lf\{\int_{\wh Q}\,d|\mu|\r\}^{1/2}<\fz,
\end{equation*}
where the supremum is taken over all cubes $Q$ of $\rn$ and $\wh Q$ denotes the tent
over $Q$; see \cite{zyl14}.

\begin{proposition}\label{p-carleson}
Let $p(\cdot)$, $s_0$ and $s$ be as in Proposition \ref{p-bmo1}.
If $f\in {\rm BMO}_{p(\cdot),L}^{s_0}(\rn)$, then the measure
$$d\mu_f(x,t):=|Q_{s,t^m}(I-P_{s_0,t^m})f(x)|^2\,\frac{dxdt}{t},
\quad\forall\ (x,t)\in\rr^{n+1}_+$$
is a $p(\cdot)$-Carleson measure on $\urn$ and there exists a positive constant $C$,
independent of $f$, such that
$\|d\mu_f\|_{p(\cdot)}\le C\|f\|_{{\rm BMO}_{p(\cdot),L}^{s_0}(\rn)}.$
\end{proposition}

\begin{proof}
Since $s\ge s_0= \lfloor \frac nm(\frac1{p_-}-1)\rfloor$ and $\theta(L)\in (n[\frac1{p_-}-1],\fz)$
with $\theta(L)$ as in \eqref{2.3x},
it follows that there exists $\vez\in(n[\frac1{p_-}-1],\theta(L))$
such that $m(s+1)>\vez$.
To prove this proposition, by definition, it suffices to show that, for any cube
$R:=R(x_R,r_R)$ for some $x_R\in\rn$ and $r_R\in(0,\fz)$,
\begin{equation}\label{carleson-1}
\frac{|R|^{1/2}}{\|\chi_R\|_{\vlp}}
\lf\{\int_{\wh R}|Q_{s,t^m}(I-P_{s_0,t^m})f(x)|^2\,\frac{dxdt}{t}\r\}^{1/2}
\ls\|f\|_{{\rm BMO}_{p(\cdot),L}^{s_0}(\rn)}.
\end{equation}

Observe that
$$I-P_{s_0,t^m}=(I-P_{s_0,t^m})\lf[I-P_{s_0,(r_R)^m}\r]+P_{s_0,(r_R)^m}(I-P_{s_0,t^m}).$$
Then the estimate \eqref{carleson-1} is a direct consequence of
\begin{eqnarray}\label{carleson-2}
&&\frac{|R|^{1/2}}{\|\chi_R\|_{\vlp}}
\lf\{\int_{\wh R}\lf|Q_{s,t^m}(I-P_{s_0,t^m})
\lf[I-P_{s_0,(r_R)^m}\r]f(x)\r|^2\,\frac{dxdt}{t}\r\}^{1/2}\\
&&\hs\ls\|f\|_{{\rm BMO}_{p(\cdot),L}^{s_0}(\rn)}\noz
\end{eqnarray}
and
\begin{equation}\label{carleson-3}
\frac{|R|^{1/2}}{\|\chi_R\|_{\vlp}}
\lf\{\int_{\wh R}|Q_{s,t^m}P_{s_0,(r_R)^m}(I-P_{s_0,t^m})f(x)|^2\,\frac{dxdt}{t}\r\}^{1/2}
\ls\|f\|_{{\rm BMO}_{p(\cdot),L}^{s_0}(\rn)}.
\end{equation}

Next we prove \eqref{carleson-2} and \eqref{carleson-3}, respectively.
To show \eqref{carleson-2}, let
$$b_1:=\lf[I-P_{s_0,(r_R)^m}\r]f\chi_{2R}\quad {\rm and}\quad
b_2:=\lf[I-P_{s_0,(r_R)^m}\r]f\chi_{\rn\backslash(2R)}.$$
By \cite[(4.25)]{jyz09}, we know that
\begin{eqnarray}\label{carleson-4}
{\rm J}:=\lf\{\int_{\wh R}|Q_{s,t^m}(I-P_{s_0,t^m})b_1(x)|^2\,
\frac{dxdt}{t}\r\}^{1/2}
\ls\|b_1\|_{L^2(\rn)},
\end{eqnarray}
which, combined with Proposition \ref{p-bmo1} and Lemma \ref{l-bigsball}, implies that
\begin{eqnarray*}
{\rm J}
&&\ls\lf\{\int_{2R}\lf|\lf[I-P_{s_0,(2r_R)^m}\r]f(x)\r|^2\,dx\r\}^{\frac12}\\
&&\quad+|R|^{1/2}\sup_{x\in 2R}|P_{s_0,(r_R)^m}f(x)-P_{s_0,(2r_R)^m}f(x)|\\
&&\ls|R|^{-1/2}\|\chi_R\|_{\vlp}\|f\|_{{\rm BMO}_{p(\cdot),L}^{s_0}(\rn)}.
\end{eqnarray*}
For $b_2$, we write
\begin{equation}\label{carleson-5}
Q_{s,t^m}(I-P_{s_0,t^m})b_2
=Q_{s,t^m}b_2-Q_{s,t^m}P_{s_0,t^m}b_2.
\end{equation}
Let $(x,t)\in \wh R$. Then, by \eqref{assump2}, \eqref{operator-k}
and Proposition \ref{p-bmo2}, we have
\begin{eqnarray*}
|Q_{s,t^m}b_2(x)|
&&\ls\int_{\rn\backslash(2R)}\frac{t^\vez}{|x-y|^{n+\vez}}
\lf|\lf[I-P_{s_0,(r_R)^m}\r]f(y)\r|\,dy\\
&&\ls\int_\rn\frac{t^\vez}{(r_R+|x-y|)^{n+\vez}}\lf|\lf[I-P_{s_0,(r_R)^m}\r]f(y)\r|\,dy\\
&&\ls (t/r_R)^\vez|R|^{-1}\|\chi_R\|_{\vlp}\|f\|_{{\rm BMO}_{p(\cdot),L}^{s_0}(\rn)},
\end{eqnarray*}
which implies that
\begin{eqnarray}\label{carleson-6}
\lf\{\int_{\wh R}|Q_{s,t^m}b_2(x)|^2\,\frac{dxdt}{t}\r\}^{1/2}
\ls|R|^{-1/2}\|\chi_R\|_{\vlp}\|f\|_{{\rm BMO}_{p(\cdot),L}^{s_0}(\rn)}.
\end{eqnarray}
On the other hand, for all $k\in\{1,\dots,s_0+1\}$, $t,\, v\in(0,\fz)$,
$f\in\cm(\rn)$ and $x\in\rn$,
let
\begin{equation}\label{carleson-6x1}
\Psi_{t,v}^kf(x):=[kv^m+t^m]^{s+1}
\lf(\lf.\frac{d^{s+1}P_\eta}{d\eta^{s+1}}\r|_{kv^m+t^m}f\r)(x).
\end{equation}
Then, by Assumption \ref{as-a} and \eqref{partial1}, we conclude that $\psi_{t,v}^k$, the kernel of
$\Psi_{t,v}^k$, satisfies that, for all $t,\, v\in(0,\fz)$ and $x,\ y\in\rn$,
\begin{equation}\label{carleson-6x2}
|\psi_{t,v}^k(x,y)|\ls \frac{v^\vez}{(v+t+|x-y|)^{n+\vez}}.
\end{equation}
From this, Proposition \ref{p-bmo2} and the fact that
\begin{equation}\label{carleson-6x3}
P_{s_0,v^m}f=\sum_{k=1}^{s_0+1}(-1)^{k+1}C_{s_0+1}^ke^{-kv^mL}f,
\end{equation}
where $C_{s_0+1}^k:=\frac{(s_0+1)!}{k!(s_0+1-k)!}$,
we deduce that, for all $(x,t)\in\wh R$,
\begin{eqnarray}\label{carleson-6x}
|Q_{s,t^m}P_{s_0,t^m}b_2(x)|
&&=\lf|\sum_{k=1}^{s_0+1}(-1)^kC_{s_0+1}^k\frac{t^{m(s+1)}}{[(k+1)t^m]^{s+1}}
\Psi_{t,t}^kb_2(x)\r|\\
&&\ls\int_{\rn\backslash(2R)}\frac{t^\vez}{(t+|x-y|)^{n+\vez}}
\lf|\lf[I-P_{s_0,(r_R)^m}\r]f(y)\r|\,dy\noz\\
&&\ls\lf(\frac{t}{r_R}\r)^{\vez}\int_{\rn}\frac{(r_R)^\vez}{(r_R+|x-y|)^{n+\vez}}
\lf|\lf[I-P_{s_0,(r_R)^m}\r]f(y)\r|\,dy\noz\\
&&\ls\lf({t}/{r_R}\r)^{\vez}\|\chi_R\|_{\vlp}
\|f\|_{{\rm BMO}_{p(\cdot),L}^{s_0}(\rn)},\noz
\end{eqnarray}
which further implies that
\begin{equation}\label{carleson-7}
\lf\{\int_{\wh R}|Q_{s,t^m}P_{s_0,t^m}b_2(x)|^2\,\frac{dxdt}{t}\r\}^{1/2}
\ls|R|^{-1/2}\|\chi_R\|_{\vlp}\|f\|_{{\rm BMO}_{p(\cdot),L}^{s_0}(\rn)}.
\end{equation}

Combining \eqref{carleson-4}, \eqref{carleson-5}, \eqref{carleson-6} and \eqref{carleson-7},
we conclude that \eqref{carleson-2} holds true.

Similarly, by \eqref{carleson-6x1}, \eqref{carleson-6x2}, \eqref{carleson-6x3}
and Proposition \ref{p-bmo2},
we also see that, for all $(x,t)\in \wh R$,
\begin{eqnarray*}
&&\lf|Q_{s,t^m}P_{s_0,(r_R)^m}(I-P_{s_0,t^m})(f)(x)\r|\\
&&\hs=\lf|\sum_{k=1}^{s_0+1}(-1)^kC_{s_0+1}^k\frac{t^{m(s+1)}}{[k(r_R)^m+t^m]^{s+1}}
\Psi_{t,r_R}^k(I-P_{s_0,t^m})(f)(x)\r|\\
&&\hs\ls\frac{t^{m(s+1)}}{[k(r_R)^m+t^m]^{s+1}}\int_{\rn}\frac{(r_R)^\vez}
{(r_R+t+|x-y|)^{n+\vez}}|(I-P_{s_0,t^m})(f)(y)|\,dy\\
&&\hs\ls\lf(\frac{t}{r_R}\r)^{m(s+1)-\vez}\int_\rn\frac{t^\vez}{(t+|x-y|)^{n+\vez}}
|(I-P_{s_0,t^m})(f)(r)|\,dy\\
&&\hs \ls t^{-n}\lf(\frac{t}{r_R}\r)^{m(s+1)-\vez}
\|\chi_{Q(x,t)}\|_{\vlp}\|f\|_{{\rm BMO}_{p(\cdot),L}^s(\rn)},
\end{eqnarray*}
which, together with Lemma \ref{l-bigsball}, implies that, for all $(x,t)\in \wh R$,
\begin{equation*}
\lf|Q_{s,t^m}P_{s_0,(r_R)^m}(I-P_{s_0,t^m})(f)(x)\r|
\ls t^{-n}\lf(\frac{t}{r_R}\r)^{m(s+1)-\vez}
\|\chi_{R}\|_{\vlp}\|f\|_{{\rm BMO}_{p(\cdot),L}^s(\rn)}.
\end{equation*}
By this and the fact that $m(s+1)>\vez$, we further conclude that
\eqref{carleson-3} holds true.
This finishes the proof of Proposition \ref{p-carleson}.
\end{proof}

\begin{proposition}\label{p-dual}
{\rm(i)} Let $q\in(1,\fz)$ and $q^\ast:=\frac q{q-1}$. Then, for all $f\in T_2^q(\urn)$
and $g\in T_2^{q^\ast}(\urn)$,
$$\int_{\urn}|f(y,t)g(y,t)|\,\frac{dydt}{t}
\le \int_\rn\ct(f)(x)\ct(g)(x)\,dx.$$

{\rm(ii)} Let $p(\cdot)\in C^{\log}(\rn)$ satisfy $p_+\in(0,1]$.
Then the dual space of $T_2^{p(\cdot)}(\urn)$ is $T_{2,\fz}^{p(\cdot)}(\urn)$ in the following
sense: for any $h\in T_{2,\fz}^{p(\cdot)}(\urn)$, the mapping
\begin{equation}\label{dual}
\ell_h(f):=\int_{\urn} h(y,t)f(y,t)\,\frac{dydt}{t}
\end{equation}
is a bounded linear functional on $T_2^{p(\cdot)}(\urn)$;
conversely, if $\ell$ is a bounded linear functional on $T_2^{p(\cdot)}(\urn)$,
then $\ell$ has the form as in \eqref{dual} with a unique $h\in T_{2,\fz}^{p(\cdot)}(\urn)$.
Moreover, $\|h\|_{T_{2,\fz}^{p(\cdot)}(\urn)}\sim\|\ell_h\|_{(T_2^{p(\cdot)}(\urn))^\ast}$
with the implicit positive constants independent of $h$.
\end{proposition}
\begin{proof}
To prove this proposition, it suffices to show (ii), since (i) was already proved in
\cite[p.\,316, Theorem 2]{cms85}.

We first show that $T_{2,\fz}^{p(\cdot)}(\urn)\subset (T_{2}^{p(\cdot)}(\urn))^\ast$.
Let $h\in T_{2,\fz}^{p(\cdot)}(\urn)$. Then, by the H\"older inequality
and Remark \ref{r-atom}(ii),
we find that, for any $(p(\cdot),\fz)$-atom $a$
supported on $\wh Q$ for some cube $Q\st\rn$,
\begin{eqnarray}\label{dual-x}
\int_{\urn}|h(x,t)a(x,t)|\,\frac{dxdt}{t}
&&\le \frac{|Q|^{1/2}}{\|\chi_Q\|_{\vlp}}
\lf\{\int_{\wh Q}|h(x,t)|^2\,\frac{dxdt}{t}\r\}^{1/2}\\
&&\le \|\mathcal{C}_{p(\cdot)}(h)\|_{L^\fz(\rn)}=\|h\|_{T_{2,\fz}^{p(\cdot)}(\urn)}.\noz
\end{eqnarray}
For any $f\in T_2^{p(\cdot)}(\urn)$, by Lemma \ref{l-tent}, we know that,
for almost every $(x,t)\in\urn$,
$f(x,t)=\sum_{j\in\nn}\lz_ja_j(x,t)$, where $\{\lz_j\}_{j\in\nn}$ and
$\{a_j\}_{j\in\nn}$ are as in Lemma \ref{l-tent} satisfying \eqref{tent-y}.
From this, \eqref{dual-x} and Remark \ref{r-tent}, we deduce that
\begin{eqnarray*}
|\ell_h(f)|
&&\le\sum_{j\in\nn}|\lz_j|\int_{\urn}|h(x,t)||a_j(x,t)|\,\frac{dxdt}{t}\\
&&\le\sum_{j\in\nn}|\lz_j|\|h\|_{T_{2,\fz}^{p(\cdot)}(\urn)}
\ls \|h\|_{T_{2,\fz}^{p(\cdot)}(\urn)}\|f\|_{T_2^{p(\cdot)}(\urn)},
\end{eqnarray*}
which implies that $\ell_h$ is a bounded linear functional on
$T_{2}^{p(\cdot)}(\urn)$
and
$$\|\ell_h\|_{({T_{2}^{p(\cdot)}(\urn)})^\ast}
\ls\|h\|_{T_{2,\fz}^{p(\cdot)}(\urn)}.$$

Next, we prove that
$(T_{2}^{p(\cdot)}(\urn))^\ast\subset T_{2,\fz}^{p(\cdot)}(\urn)$.
Let $\ell\in (T_{2}^{p(\cdot)}(\urn))^\ast$.
For all $k\in\nn$, let $\wz O_k:=\{(x,t)\in\urn:\ |x|\le k, 1/k\le t\le k\}$.
Then $\{\wz O_k\}_{k\in\nn}$ is a family of compact sets of $\urn$ and
$\urn=\bigcup_{k\in\nn}\wz O_k$. Observe that, for each $k\in\nn$, if
$f\in L^2(\urn)$ with $\supp f\subset \wz O_k$, then
$\supp \ct(f)\subset O_k^\ast:=\{x\in\rn:\ |x|\le 2k\}$. It follows, from the
H\"older inequality, that
\begin{eqnarray*}
\int_{O_k^\ast}\ct(f)(x)\,dx
&&\le |O_k^\ast|^{1/2}\lf\{\int_{O_k^\ast}\int_{\Gamma(x)}|f(y,t)|^2\frac{dydt}{t^{n+1}}
dx\r\}^{1/2}\\
&&\ls|O_k^\ast|^{1/2}\lf\{\int_{\wz O_k}|f(y,t)|^2\frac{dydt}{t}\r\}^{1/2}
\sim |O_k^\ast|^{1/2}\|f\|_{L^2(\wz O_k)}.
\end{eqnarray*}
By this and the fact that $p_+\in(0,1]$, we further find that
\begin{eqnarray*}
\int_\rn\lf[\frac{\ct(f)(x)}{|O_k^\ast|^{-\frac12}\|f\|_{L^2(\wz O_k)}}\r]^{p(x)}\,dx
\le \int_{O_k^\ast}\lf[1+\frac{|O_k^\ast|^{\frac12}\ct(f)(x)}
{\|f\|_{L^2(\wz O_k)}}\r]^{p(x)}\,dx\ls|O_k^\ast|,
\end{eqnarray*}
which implies that $\|f\|_{T_2^{p(\cdot)}(\urn)}\le C_{(k)}\|f\|_{L^2(\wz O_k)}$,
where $C_{(k)}$ is a positive constant depending on $k$.
Thus, $\ell$ also induces a bounded linear functional on $L^2(\wz O_k)$.
By the Riesz theorem, there exists a unique $h_k\in L^2(\wz O_k)$ such that,
for all $f\in L^2(\wz O_k)$,
\begin{equation*}
\ell(f)=\int_{\urn} f(x,t)h_k(x,t)\,\frac{dxdt}t.
\end{equation*}
Obviously, $h_{k+1}\chi_{\wz O_k}=h_k$ for all $k\in\nn$. Let
$$h:=h_1\chi_{\wz O_1}+\sum_{k=2}^\fz h_k\chi_{\wz O_{k}\backslash \wz O_{k-1}}.$$
Then $h\in L_{\loc}^2(\urn)$ and, for any $f\in L^2(\urn)$ having compact
support,
$$\ell(f)=\int_{\urn}f(y,t)h(y,t)\,\frac{dydt}{t}.$$

Now, for any $f\in T_2^{p(\cdot)}(\urn)$, by Lemma \ref{l-tent}, we have
$f(x,t)=\sum_{j\in\nn}\lz_ja_j(x,t)$ for almost every $(x,t)\in\urn$,
where $\{\lz_j\}_{j\in\nn}$ and
$\{a_j\}_{j\in\nn}$ are as in Lemma \ref{l-tent} satisfying \eqref{tent-y}.
For all $N\in\nn$, let $f_N:=\sum_{j=1}^N\lz_ja_j$. Then
$f_N\to f$ in $T_2^{p(\cdot)}(\urn)$ as $N\to\fz$ due to
Corollary \ref{c-tent1}.
Moreover, it is easy to see that
$f_N\in L^2(\urn)$ having compact support and hence
$$\ell(f_N)=\int_{\urn}f_N(y,t)h(y,t)\,\frac{dydt}{t}.$$
Observer that, for all $N\in\nn$,
$$|f_N|\le\sum_{j=1}^N|\lz_j||a_j|\le \sum_{j\in\nn}|\lz_j||a_j|$$
and, by \eqref{dual-x} and
Remark \ref{r-tent}, we find that
\begin{eqnarray*}
&&\sum_{j\in\nn}|\lz_j|\int_{\urn}|h(x,t)||a_j(x,t)|\,\frac{dxdt}{t}\\
&&\hs\ls\|h\|_{T_{2,\fz}^{p(\cdot)}(\urn)}\sum_{j\in\nn}|\lz_j|
\ls\|h\|_{T_{2,\fz}^{p(\cdot)}(\urn)}\cb(\{\lz_ja_j\}_{j\in\nn})\\
&&\hs\ls\|h\|_{T_{2,\fz}^{p(\cdot)}(\urn)}\|f\|_{T_2^{p(\cdot)}(\urn)}.
\end{eqnarray*}
Therefore, from the dominated convergence theorem, we deduce that
\begin{eqnarray*}
\ell(f)=\lim_{N\to\fz}\ell(f_N)=\int_{\urn}f(y,t)h(y,t)\,\frac{dydt}{t}.
\end{eqnarray*}

To complete the proof of this proposition, it remains to show that
$h\in T_{2,\fz}^{p(\cdot)}(\urn)$. Indeed, for any cube $Q\subset\rn$ and $j\in\nn$,
let $R_{Q,j}:=\wh Q\cap\{(x,t)\in\urn:\ t\ge1/j\}$ and
$$\eta_j:=\frac{|Q|^{1/2}\overline{h}\chi_{R_{Q,j}}}
{\|\chi_Q\|_{\vlp}\|h\chi_{R_{Q,j}}\|_{L^2(\urn)}}.$$
Then, by the Minkowski inequality, we find that
\begin{eqnarray*}
\|\eta_j\|_{T_2^2(\urn)}
&&\sim\lf\{\int_\rn\int_{\Gamma(x)}|\eta_j(y,t)|^2\,\frac{dydt}{t^{n+1}}
\,dx\r\}^{\frac12}\\
&&\ls\lf\{\int_{R_{Q,j}}|\eta_j(y,t)|^2\,\frac{dydt}{t}\r\}^{\frac12}
\ls|Q|^{1/2}\|\chi_Q\|_{\vlp}^{-1},
\end{eqnarray*}
namely, $\eta_j$ is a $(p(\cdot),2)$-atom up to a positive constant multiple.
From this, the Fatou lemma
and Corollary \ref{c-tent}, we further deduce that
\begin{eqnarray*}
&&\frac{|Q|^{1/2}}{\|\chi_Q\|_{\vlp}}\lf\{\int_{\wh Q}|h(y,t)|^2\,\frac{dydt}{t}\r\}^{1/2}\\
&&\hs\le\liminf_{j\to\fz}
\int_{\urn}\frac{|Q|^{\frac12}\overline{h(y,t)}h(y,t)\chi_{R_{Q,j}}(y,t)}
{\|\chi_Q\|_{\vlp}\|h\chi_{R_{Q,j}}\|_{L^2(\urn)}}\,\frac{dydt}{t}
=\liminf_{j\to\fz}\ell(\eta_j)\\
&&\hs\ls\liminf_{j\to\fz}\|\ell\|_{(T_2^{p(\cdot)}(\urn))^\ast}
\|\eta_j\|_{T_{2}^{p(\cdot)}(\urn)}\ls\|\ell\|_{(T_2^{p(\cdot)}(\urn))^\ast},
\end{eqnarray*}
which, together with the arbitrariness of cubes $Q$,
implies that $h\in T_{2,\fz}^{p(\cdot)}(\urn)$ and
$$\|h\|_{T_{2,\fz}^{p(\cdot)}(\urn)}\ls
\|\ell\|_{(T_{2}^{p(\cdot)}(\urn))^\ast}.$$
This finishes the proof of Proposition \ref{p-dual}.
\end{proof}

To prove Theorem \ref{t-dual}, we also need the following estimate.

\begin{lemma}\label{l-dual}
Let $p(\cdot)$ and $s_0$ be as in Proposition \ref{p-bmo1}.
Then there exists a positive constant $C$ such that,
for all $f\in L^2(\rn)$ satisfying $\supp f\st Q:=Q(x_Q,r_Q)$ for some $x_Q\in\rn$
and $r_Q\in(0,\fz)$,
\begin{equation*}
\lf\|(I-P_{s_0,r_Q^m})f\r\|_{H_L^{p(\cdot)}(\rn)}
\le C|Q|^{-1/2}\|\chi_Q\|_{\vlp}\|f\|_{L^2(\rn)}.
\end{equation*}
\end{lemma}

\begin{proof}
Obviously, we have
\begin{eqnarray}\label{dual2-0}
&&\lf\|S_L([I-P_{s_0,r_Q^m}]f)\r\|_{\vlp}\\
&&\hs\ls\lf\|S_L([I-P_{s_0,r_Q^m}]f)\chi_{4Q}\r\|_{\vlp}
+\lf\|S_L([I-P_{s_0,r_Q^m}]f)\chi_{\rn\setminus(4Q)}\r\|_{\vlp}\noz\\
&&\hs=:{\rm J}_1+{\rm J}_2.\noz
\end{eqnarray}
By the boundedness of $S_L$ in $L^2(\rn)$ (see \eqref{SL-bounded})
and Remark \ref{r-operator}(iii), we see that
$$\|S_L([I-P_{s_0,r_Q^m}]f)\|_{L^2(\rn)}\ls\|f\|_{L^2(\rn)},$$
which, together with Lemmas \ref{l-estimate} and \ref{l-bigsball}, implies that
\begin{equation}\label{dual2-1}
{\rm J}_1\ls |Q|^{-1/2}\|\chi_Q\|_{\vlp}\|f\|_{L^2(\rn)}.
\end{equation}
To deal with the term ${\rm J}_2$, since $p_-\in (\frac n{n+\theta(L)},1]$
with $p_-$ and $\theta(L)$ as in \eqref{2.1x} and \eqref{2.3x}, respectively,
we choose $\vez\in(0,\theta(L))$ such that $p_-\in (\frac n{n+\vez},1]$.
Notice that, for all $x\notin 4Q$,
\begin{equation*}
S_L(I-P_{s_0,r_Q^m})(f)(x)\ls \frac{(r_Q)^{\frac n2+\vez}}{|x-x_Q|^{n+\ez}}\|f\|_{L^2(\rn)};
\end{equation*}
see \cite[(4.20)]{jyz09}. Then, by this, Lemma \ref{l-bigsball}
and the fact that $\vez\in(n(\frac1{p_-}-1),\theta(L))$, we further know that,
for any $r\in(0,p_-)$,
\begin{eqnarray}\label{dual2-2}
{\rm J}_2
&&\sim \lf\|\sum_{k=0}^\fz S_L(I-P_{s_0,r_Q^m})(f)\chi_{(4^{k+1}Q)\backslash (4^kQ)}
\r\|_{\vlp}\\
&&\ls \lf\{\sum_{k=0}^\fz \lf\|\frac{(r_Q)^{\frac n2+\vez}}{|\cdot-x_Q|^{n+\ez}}
\chi_{(4^{k+1}Q)\backslash (4^kQ)}\r\|_{\vlp}^r\r\}^{\frac1r}\|f\|_{L^2(\rn)}\noz\\
&&\ls \lf\{\sum_{k=0}^\fz 4^{-k(n+\vez)r}\|\chi_{4^kQ}\|_{\vlp}\r\}^{\frac1r}
|Q|^{-\frac12}\|f\|_{L^2(\rn)}\noz\\
&&\ls \lf\{\sum_{k=0}^\fz 4^{-k(n+\vez-n/p_-)r}\r\}^{\frac1r}
\|\chi_Q\|_{\vlp}|Q|^{-\frac12}\|f\|_{L^2(\rn)}\noz\\
&&\sim |Q|^{-\frac12}\|\chi_Q\|_{\vlp}\|f\|_{L^2(\rn)}.\noz
\end{eqnarray}

Combining the estimates \eqref{dual2-0}, \eqref{dual2-1} and \eqref{dual2-2},
we conclude the desired result and then complete the proof of Lemma \ref{l-dual}.
\end{proof}

For all $s\in[s_0,\fz)$ with $s_0$ as in \eqref{3.6x},
$t\in(0,\fz)$, $f\in \cm(\rn)$ and $x\in\rn$, let
$$P_{s,t}^\ast f(x):=f(x)-(I-e^{-tL^\ast})^{s+1}f(x)
\quad{\rm and}\quad
Q_{s,t}^\ast f(x):=(tL^\ast)^{s+1}e^{-tL^\ast}f(x),$$
where $L^\ast$ denotes the adjoint operator of $L$ in $L^2(\rn)$.
Suppose that $\az$ is a $(p(\cdot),s,L)$-molecule.
Then, by Theorem \ref{t-mol}(ii), we see that $\az\in H_L^{p(\cdot)}(\rn)$
and hence $G:=Q_{t^m}\az\in T_2^{p(\cdot)}(\urn)$.
Let $f\in\cm(\rn)$ be such that
$$\mu_f(x,t):=|Q_{s,t^m}^\ast(I-P_{s_0,t^m}^\ast)f(x)|^2\,\frac{dxdt}t,
\quad\forall\ (x,t)\in\rr^{n+1}_+$$
 is a $p(\cdot)$-Carleson measure on $\urn$ and, for all $(x,t)\in\urn$, let
$$F(x,t):=Q_{s,t^m}^\ast(I-P_{s_0,t^m}^\ast)f(x).$$
Then
$\|F\|_{T_{2,\fz}^{p(\cdot)}(\urn)}\ls \|\mu_f\|_{p(\cdot)}<\fz$. From this
and Proposition \ref{p-dual}(ii), we deduce that the integral
$$J(F,G):=\int_{\urn}F(x,t)G(x,t)\,\frac{dxdt}t$$
converges absolutely and hence
$$\int_{\urn}\lf|Q_{t^m}\az(x)
Q_{s,t^m}^\ast(I-P_{s_0,t^m}^\ast)f(x)\r|\,\frac{dxdt}t<\fz.$$
Indeed, by an argument similar to that used in the
proof of \cite[Proposition 5.1]{dy05}, we have the following technical lemma,
the details being omitted.

\begin{lemma}\label{l-3.26x}
Let $p(\cdot)\in C^{\log}(\rn)$, $s_0$ and $s$ be as in Proposition
\ref{p-bmo1}.
Suppose that $\az$ is a $(p(\cdot),s,L)$-molecule and $f\in\cm(\rn)$ satisfies that
$\mu_f(x,t):=|Q_{s,t^m}^\ast(I-P_{s_0,t^m}^\ast)f(x)|^2\,\frac{dxdt}t$ for all
$(x,t)\in \urn$ is a
$p(\cdot)$-Carleson measure on $\urn$. Then
\begin{equation*}
\int_\rn f(x)\az(x)\,dx=C_{(m,s)}\int_{\urn}Q_{t^m}\az(x)
Q_{s,t^m}^\ast(I-P_{s_0,t^m}^\ast)f(x)\,\frac{dxdt}t,
\end{equation*}
where $C_{(m,s)}$ is as in \eqref{pi-def}.
\end{lemma}

We are now ready to prove Theorem \ref{t-dual}.

\begin{proof}[Proof of Theorem \ref{t-dual}]
We first prove (i). Let $g\in {\rm BMO}_{p(\cdot),L^\ast}^{s_0}(\rn)$ and
$f\in H_{L,{\rm fin}}^{p(\cdot)}(\rn)$. Then $f$ has an expression
$f=\sum_{j=1}^N\lz_ja_j$, where $N\in\nn$,
$\{\lz_j\}_{j=1}^N\st\cc$ and $\{\az_j\}_{j=1}^N$
are $(p(\cdot),s,L)$-molecules associated with cubes $\{R_j\}_{j=1}^N$ of $\rn$
satisfying
\begin{equation*}
\lf\|\lf\{\sum_{j=1}^N\lf[\frac{|\lz_j|}{\|\chi_{R_j}\|_{\vlp}}\chi_{R_j}
\r]^{\underline{p}}
\r\}^{\frac1{\underline{p}}}\r\|_{\vlp}\ls \|f\|_{H_{L,{\rm fin}}^{p(\cdot)}(\rn)}.
\end{equation*}
For each $j\in\nn$, since $\az_j\in H_L^{p(\cdot)}(\rn)$, it follows that
$Q_{t^m}\az_j\in T_2^{p(\cdot)}(\urn)$. By this and Corollary \ref{c-tent},
we know that, for any $j\in\nn$, there exist $\{\lz_j^k\}_{k\in\nn}\st\cc$ and
a sequence $\{a_j^k\}_{k\in\nn}$ of $(p(\cdot),\fz)$-atoms such that
$Q_{t^m}\az_j=\sum_{k\in\nn}\lz_j^ka_j^k$ almost everywhere on $\urn$,
$\supp a_j^k\st \wh R_j^k$ with some cube $R_j^k\st\rn$ for all $k\in\nn$,
and
$$\cb(\{\lz_j^ka_j^k\}_{k\in\nn})
\ls \|Q_{t^m}\az_j\|_{T_2^{p(\cdot)}(\urn)}
\sim\|\az_j\|_{H_L^{p(\cdot)}(\rn)}.$$
Thus, from Lemma \ref{l-3.26x}, the H\"older inequality,
Proposition \ref{p-carleson}, Remarks \ref{r-atom}(ii) and \ref{r-tent},
we deduce that
\begin{eqnarray*}
&&\lf|\int_\rn \az_j(x)g(x)\,dx\r|\\
&&\hs\sim\lf|\int_{\urn}Q_{t^m}\az_j(y)Q_{s,t^m}^\ast
(I-P_{s_0,t^m}^\ast)g(y)\r|\,\frac{dydt}{t}\\
&&\hs\ls\sum_{k\in\nn}\int_{\wh R_j^k}|\lz_j^ka_j^k(y,t)
Q_{s,t^m}^\ast(I-P_{s_0,t^m}^\ast)g(y)|\,\frac{dydt}{t}\\
&&\hs\ls\sum_{k\in\nn}|\lz_j^k|
\lf\{\int_{\wh R_j^k}|a_j^k(y,t)|^2\,\frac{dydt}{t}\r\}^{\frac12}
|R_j^k|^{\frac12}\|\chi_{R_j^k}\|_{\vlp}\|g\|_{{\rm BMO}_{p(\cdot),L^\ast}^{s_0}(\rn)}\\
&&\hs\ls\sum_{k\in\nn}|\lz_j^k|\|g\|_{{\rm BMO}_{p(\cdot),L^\ast}^{s_0}(\rn)}
\ls\cb(\{\lz_j^ka_j^k\}_{k\in\nn})\|g\|_{{\rm BMO}_{p(\cdot),L^\ast}^{s_0}(\rn)}\\
&&\hs\ls\|g\|_{{\rm BMO}_{p(\cdot),L^\ast}^{s_0}(\rn)}\|\az_j\|_{H_L^{p(\cdot)}(\rn)}
\ls\|g\|_{{\rm BMO}_{p(\cdot),L^\ast}^{s_0}(\rn)}.
\end{eqnarray*}
By this and Remark \ref{r-tent}, we further obtain
\begin{eqnarray*}
\lf|\int_\rn f(x)g(x)\,dx\r|
&&\le \sum_{j=1}^N|\lz_j|\lf|\int_\rn\az_j(x)g(x)\,dx\r|
\ls\sum_{j=1}^N|\lz_j|\|g\|_{{\rm BMO}_{p(\cdot),L^\ast}^{s_0}(\rn)}\\
&&\ls \lf\|\lf\{\sum_{j=1}^N\lf[\frac{|\lz_j|}{\|\chi_{R_j}\|_{\vlp}}
\chi_{R_j}\r]^{\underline{p}}\r\}^{\frac1{\underline{p}}}\r\|_{\vlp}
\|g\|_{{\rm BMO}_{p(\cdot),L^\ast}^{s_0}(\rn)}\\
&&\ls \|f\|_{H_L^{p(\cdot)}(\rn)}\|g\|_{{\rm BMO}_{p(\cdot),L^\ast}^{s_0}(\rn)}.
\end{eqnarray*}
Therefore, by Corollary \ref{c-dense} and a density argument,
we conclude that $\ell_g$ is a bounded linear functional on $H_L^{p(\cdot)}(\rn)$
and
$\|\ell_g\|_{(H_L^{p(\cdot)}(\rn))^\ast}
\ls \|g\|_{{\rm BMO}_{p(\cdot),L^\ast}^{s_0}(\rn)}$.

Next we show (ii). For any $\eta\in T_2^{p(\cdot)}(\urn)$, by Proposition \ref{p-tent2}(ii),
we know that $\pi_L(\eta)\in H_L^{p(\cdot)}(\rn)$ and hence, for any
$\ell\in (H_L^{p(\cdot)}(\rn))^\ast$, we have
\begin{eqnarray*}
\lf|(\ell\circ\pi_L)(\eta)\r|
&&=|\ell(\pi_l(\eta))|\ls\|\ell\|_{(H_L^{p(\cdot)}(\rn))^\ast}
\|\pi_L(\eta)\|_{H_L^{p(\cdot)}(\rn)}\\
&&\ls\|\ell\|_{(H_L^{p(\cdot)}(\rn))^\ast}\|\pi_L\|_{T_2^{p(\cdot)}(\urn)\to
H_L^{p(\cdot)}(\rn)}\|\eta\|_{T_2^{p(\cdot)}(\urn)}.
\end{eqnarray*}
In other words, $\ell\circ\pi_L$ is a bounded linear functional on
$T_2^{p(\cdot)}(\urn)$. Thus, by Proposition \ref{p-dual}(ii), we find that
there exists a function
$h\in T_{2,\fz}^{p(\cdot)}(\rn)$ such that, for all $\eta\in T_2^{p(\cdot)}(\urn)$,
\begin{equation}\label{dual2-x}
(\ell\circ\pi_L)(\eta)=\int_{\urn}\eta(x,t)h(x,t)\,\frac{dxdt}{t}.
\end{equation}
On the other hand, by Remark \ref{r-3.14x},
we see that, for all $f\in H_L^{p(\cdot)}(\rn)\cap L^2(\rn)$,
$f=\pi_L(Q_{t^m}f)$ in $L^2(\rn)$. From this and \eqref{dual2-x}, we deduce that
\begin{eqnarray*}
\ell(f)&&=(\ell\circ\pi_L)(Q_{t^m}f)=\int_{\urn}h(x,t)Q_{t^m}f(x)\,\frac{dxdt}{t}\\
&&=\int_\rn\lf\{\int_0^\fz(Q_{t^m}^\ast h)(x,t)\,\frac{dt}t\r\}f(x)\,dx
=:\int_\rn g(x)f(x)\,dx.
\end{eqnarray*}

To complete the proof of Theorem \ref{t-dual}, it remains to prove that
$g\in {\rm BMO}_{p(\cdot),L^\ast}^{s_0}(\rn)$.
For any $Q:=Q(x_Q,r_Q)\st\rn$ for some $x_Q\in\rn$ and $r_Q\in(0,\fz)$,
by Lemma \ref{l-dual}, we conclude that
\begin{eqnarray*}
\lf\{\int_Q|g(x)-P_{s_0,r_Q^m}^\ast g(x)|^2\,dx\r\}^{\frac 12}
&&=\sup_{\|u\|_{L^2(Q)}\le1}
\lf|\int_\rn\lf[g(x)-P_{s_0,r_Q^m}^\ast g(x)\r]u(x)\,dx\r|\\
&&=\sup_{\|u\|_{L^2(Q)}\le1}
\lf|\int_\rn g(x)\lf[(I-P_{s_0,r_Q^m})u(x)\r]\,dx\r|\\
&&=\sup_{\|u\|_{L^2(Q)}\le1}\lf|\ell([I-P_{s_0,r_Q^m}]u)\r|\\
&&\ls\|\ell\|_{(H_L^{p(\cdot)}(\rn))^\ast}
\sup_{\|u\|_{L^2(Q)}\le1}\lf\|(I-P_{s_0,r_Q^m})u\r\|_{H_L^{p(\cdot)}(\rn)}\\
&&\ls \|\ell\|_{(H_L^{p(\cdot)}(\rn))^\ast}|Q|^{-1/2}\|\chi_Q\|_{\vlp}.
\end{eqnarray*}
From this, we deduce that $g\in {\rm BMO}_{p(\cdot),L^\ast}^{s_0}(\rn)$ and
$\|g\|_{{\rm BMO}_{p(\cdot),L^\ast}^{s_0}(\rn)}
\ls\|\ell\|_{(H_L^{p(\cdot)}(\rn))^\ast}$,
which complete the proof of Theorem \ref{t-dual}.
\end{proof}

From Proposition \ref{p-carleson}, Theorem \ref{t-dual}
and an argument similar to that used in the proof of \cite[Theorem 4.5]{jyz09},
we easily deduce
the following characterization of the BMO-type spaces, the details being omitted.
\begin{corollary}
Let $p(\cdot)\in C^{\log}(\rn)$, $s_0$ and $s$ be as in Proposition \ref{p-bmo1}.
Then $g\in {\rm BMO}_{p(\cdot),L^\ast}^{s_0}(\rn)$ if and only if
$g\in\cm(\rn)$ and $|Q_{s,t^m}^\ast(I-P_{s_0,t^m}^\ast)g(x)|^2\,\frac{dxdt}t$
for all $(x,t)\in\urn$ is a $p(\cdot)$-Carleson measure. Moreover,
$$\|g\|_{{\rm BMO}_{p(\cdot),L^\ast}^{s_0}(\rn)}
\sim \lf\|Q_{s,t^m}^\ast(I-P_{s_0,t^m}^\ast)g\r\|_{T_{2,\fz}^{p(\cdot)}(\urn)}$$
with the implicit positive constants independent of $g$.
\end{corollary}

\section{Applications\label{s5}}
\hskip\parindent
In this section, we give out two applications
of the molecular characterizations of the spaces $H_L^{p(\cdot)}(\rn)$ established in
Theorem \ref{t-mol}.
One is to investigate the coincidence between the spaces
$H_L^{p(\cdot)}(\rn)$ and $\vhs$,
where $\vhs$ denotes the Hardy space with variable exponent introduced by Nakai
and Sawano in \cite{ns12}.
Another is to study the boundedness of the fractional integral
$L^{-\gamma}$ on $H_L^{p(\cdot)}(\rn)$.

\subsection{The coincidence between $H_L^{p(\cdot)}(\rn)$ and $H^{p(\cdot)}(\rn)$}
\hskip\parindent
We begin with recalling the definition of the Hardy space with variable exponent
introduced in \cite{ns12}.
Let $\cs(\rn)$ be the \emph{space of all Schwartz functions} and
$\cs'(\rn)$ its \emph{topological dual space}. For any $N\in\nn$, let
$$\cf_N(\rn):=\lf\{\psi\in\cs(\rn):\ \sum_{\beta\in\zz_+^n,\,|\beta|\le N}
\sup_{x\in\rn}(1+|x|)^N|D^\beta\psi(x)|\le1\r\},$$
where, for all $\beta:=(\beta_1,\dots,\beta_n)\in\zz_+^n$,
$|\beta|:=\beta_1+\cdots+\beta_n$ and
$D^\beta:=(\frac{\partial}{\partial x_1})^{\beta_1}
\cdots (\frac{\partial}{\partial x_n})^{\beta_n}$.
Then, for all $f\in\cs'(\rn)$, the \emph{grand maximal function}
$f^\ast_N$ is defined by setting, for all $x\in\rn$,
$$f^\ast_N(x):=\sup\lf\{|f\ast \psi_t(x)|:\ t\in(0,\fz)\ {\rm and}\
\psi\in\cf_N(\rn)\r\},$$
where, for all $t\in(0,\fz)$ and $\xi\in\rn$, $\psi_t(\xi):=t^{-n}\psi(\xi/t)$.

\begin{definition}
Let $p(\cdot)\in C^{\log}(\rn)$ and
$N\in (\frac n{p_-}+n+1,\fz)$. Then the \emph{Hardy space with variable
exponent $p(\cdot)$}, denoted by $\vhs$, is defined to be the set of all
$f\in\cs'(\rn)$ such that $f^\ast_N\in\vlp$, equipped with the quasi-norm
$\|f\|_{\vhs}:=\|f^\ast_N\|_{\vlp}$.
\end{definition}

\begin{remark}
In \cite[Theorem 3.3]{ns12}, it was proved that the space $\vhs$ is independent
of $N$ as long as $N$ is sufficiently large. Although the range of
$N$ is not presented explicitly in \cite[Theorem 3.3]{ns12},
it was pointed out in \cite[Remark 1.3(ii)]{zyl14} that $N\in (\frac n{p_-}+n+1,\fz)$
does the work.
\end{remark}

In what follows, suppose that $L$ is a linear operator of type $\nu$ on $L^2(\rn)$
with $\nu\in(0,\frac \pi2)$. Then it generates an analytic semigroup
$\{e^{-zL}\}_{z}$, where $z\in\cc$ satisfies $0\le |\arg(z)|<\frac \pi2-\nu$.
Following \cite{yan08}, we assume that the kernels of $\{e^{-tL}\}_{t>0}$,
$\{p_t\}_{t>0}$,
satisfy the following conditions:
there exist positive
constants $C$, $m$ and $\tau\in(n(\frac1{p_-}-1),1]$ such that, for all $t\in(0,\fz)$ and $x,\,y,\,h\in\rn$,
\begin{equation}\label{5.1x1}
|p_t(x,y)|\le C\frac{t^{1/m}}{(t^{1/m}+|x-y|)^{n+1}},
\end{equation}
\begin{equation}\label{5.1x2}
|p_t(x+h,y)-p_t(x,y)|+|p_t(x,y+h)-p_t(x,y)|
\le C\frac{t^{1/m}}{(t^{1/m}+|x-y|)^{n+1+\gamma}}|h|^\tau
\end{equation}
when $2|h|\le t^{1/m}+|x-y|$, and
\begin{equation}\label{5.1x3}
\int_\rn p_t(z,y)\,dz=1=\int_\rn p_t(x,z)\,dz.
\end{equation}

\begin{theorem}\label{t-5.2x}
Let $L$ be a linear operator of type $v$ on $L^2(\rn)$
with $v\in(0,\frac \pi2)$ and its heat kernel satisfies
\eqref{5.1x1}, \eqref{5.1x2} and
\eqref{5.1x3}. Assume that $p(\cdot)\in C^{\log}(\rn)$ satisfies $p_+\in(0,1]$,
$p_-\in (\frac n{n+1},1]$ and $\frac 2{p_-}-\frac1{p_+}<\frac{n+1}n$,
where $p_-$ and $p_+$ are as in \eqref{2.1x}.
Then $\vhs$ and $H_L^{p(\cdot)}(\rn)$
coincide with equivalent quasi-norms.
\end{theorem}

\begin{remark}
Obviously, if $L=-\Delta$, then its heat kernel satisfies
\eqref{5.1x1}, \eqref{5.1x2} and
\eqref{5.1x3}. It was also pointed out by Yan \cite[p.\,4405, Remark]{yan08}
that the assumptions \eqref{5.1x1}, \eqref{5.1x2} and
\eqref{5.1x3} are satisfied by the divergence form operator
$L:=-{\rm div}\,(A\nabla)$ when $A$ has real entries, or when the dimension
$n=1$ or 2 in the case of complex entries; see also \cite{dy05,dy05cpam,jyz09}
for some other examples.
\end{remark}

To prove Theorem \ref{t-5.2x}, we need the atomic characterization of $\vhs$.
Let $p(\cdot)\in\cp(\rn)$, $q\in[1,\fz]\cap(p_+,\fz]$
and $d:=\max\{0,\lfloor n(1/p_--1)\rfloor\}$. Recall that
a function $a$ on $\rn$ is called a \emph{$(p(\cdot),q,d)$-atom} if $a$ satisfies

(i) $\supp a\st R$ for some cube $R\st \rn$;

(ii) $\|a\|_{L^q(\rn)}\le\frac{|R|^{1/q}}{\|\chi_R\|_{\vlp}}$;

(iii) $\int_\rn a(x)x^\beta\,dx=0$ for all $\beta\in\zz_+^n$ with
$|\bz|\le d$.

\begin{definition}\label{d-atomh}
Let $p(\cdot)\in C^{\log}(\rn)$, $q\in[1,\fz]\cap(p_+,\fz]$
and $d:=\max\{0,\lfloor n(1/p_--1)\rfloor\}$ with $p_-$ and $p_+$ as in \eqref{2.1x}.
Then the \emph{atomic Hardy space}
$H_{\rm at}^{p(\cdot),q}(\rn)$ is defined to be the set of all $f\in\cs'(\rn)$
such that $f$ can be written as $f=\sum_{j\in\nn}\lz_ja_j$ in $\cs'(\rn)$,
where $\{\lz_j\}_{j\in\nn}\st\cc$ and $\{a_j\}_{j\in\nn}$ are $(p(\cdot),q,d)$-atoms
satisfying that, for each $j\in\nn$,
$\supp a_j\st R_j$ for some cube $R_j\st\rn$
and $\ca(\{\lz_j\}_{j\in\nn},\{R_j\}_{j\in\nn})<\fz$,
where $\ca(\{\lz_j\}_{j\in\nn},\{R_j\}_{j\in\nn})$ is as in \eqref{3.1y} with
$\{Q_j\}_{j\in\nn}$ replaced by $\{R_j\}_{j\in\nn}$.

Moreover, for any $f\in H_{\rm at}^{p(\cdot),q}(\rn)$, its quasi-norm
is defined by
$$\|f\|_{H_{\rm at}^{p(\cdot),q}(\rn)}
:=\inf\{ \ca(\{\lz_j\}_{j\in\nn},\{R_j\}_{j\in\nn})\},$$
where the infimum is taken over all admissible decompositions of $f$ as above.
\end{definition}

The following lemma was originally established by Nakai and Sawano in \cite[Theorem 4.6]{ns12}
and further improved by Sawano in \cite[Theorem 1.1]{Sa13}.

\begin{lemma}\label{l-4.4x}
Let $p(\cdot)\in C^{\log}(\rn)$ and $q\in[1,\fz]\cap(p_+,\fz]$ with
$p_+$ as in \eqref{2.1x}.
Then the spaces $\vhs$ and $H_{\rm at}^{p(\cdot),q}(\rn)$ coincide with equivalent
quasi-norms.
\end{lemma}
\begin{proof}[Proof of Theorem \ref{t-5.2x}]
To prove this theorem, by Lemma \ref{l-4.4x},
it suffices to show that $H_{\rm at}^{p(\cdot),2}(\rn)$ and
$H_L^{p(\cdot)}(\rn)$ coincide with equivalent quasi-norms.
Since $p_-\in(\frac n{n+1},1]$, it follows that
$d=\max\{0,\lfloor n(1/p_--1)\rfloor\}=0$ in this case.

We first show that $H_{\rm at}^{p(\cdot),2}(\rn)\st H_L^{p(\cdot)}(\rn)$.
To this end, let $q_t$ be the kernel of the operator $Q_t$.
Then, by \cite[Lemma 6.10]{dy05}
(see also \cite[p.\,4404]{yan08}), we find that,
for any $\gamma\in(n[\frac 1{p_-}-1],\tau)$ and $\delta\in(0,1)$,
there exists a positive constant
$C$ such that, for all $t\in(0,\fz)$ and $x,\,y,\,h\in\rn$,
\begin{equation}\label{5.2x1}
|q_t(x,y)|\le C\frac{t^{\delta/m}}{(t^{1/m}+|x-y|)^{n+\delta}},
\end{equation}
\begin{equation}\label{5.2x2}
|q_t(x+h,y)-q_t(x,y)|+|q_t(x,y+h)-q_t(x,y)|
\le C\frac{t^{\delta/m}}{(t^{1/m}+|x-y|)^{n+\delta+\gz}}|h|^\gz
\end{equation}
when $2|h|\le t^{1/m}+|x-y|$, and
\begin{equation}\label{5.3x3}
\int_\rn q_t(z,y)\,dz=1=\int_\rn q_t(x,z)\,dz.
\end{equation}

Let $f\in H_{\rm at}^{p(\cdot),2}(\rn)$. Then, by Definition \ref{d-atomh}, we
see that $f$ has an atomic decomposition $f=\sum_{j\in\nn}\lz_ja_j$,
where $\{\lz_j\}_{j\in\nn}\st\cc$ and $\{a_j\}_{j\in\nn}$
are $(p(\cdot),2,0)$-atoms such that, for each $j\in\nn$,
$\supp a_j\st R_j$ with some cube $R_j\st\rn$, and
\begin{equation}\label{5.3x3x}
\cb(\{\lz_ja_j\}_{j\in\nn})\ls \|f\|_{H_{\rm at}^{p(\cdot),2}(\rn)}.
\end{equation}
Thus, we have
\begin{eqnarray*}
\|S_L(f)\|_{\vlp}
&&\le \lf\|\sum_{j\in\nn}|\lz_j|S_L(a_j)\r\|_{\vlp}\\
&&\le \lf\|\sum_{j\in\nn}|\lz_j|S_L(a_j)\chi_{4R_j}\r\|_{\vlp}+
\lf\|\sum_{j\in\nn}|\lz_j|S_L(a_j)\chi_{(4R_j)^\complement}\r\|_{\vlp}\\
&&=:{\rm I}+{\rm II}.
\end{eqnarray*}

For I, since, due to \eqref{SL-bounded},
$\|S_L(a_j)\|_{L^2(\rn)}\ls \|a_j\|_{L^2(\rn)}
\ls \frac{|R_j|^{1/2}}{\|\chi_{R_j}\|_{\vlp}}$,
it follows, from Lemma \ref{l-estimate}, that
${\rm I}\ls\cb(\{\lz_ja_j\}_{j\in\nn})\ls\|f\|_{H_{\rm at}^{p(\cdot),2}(\rn)}$.

 Next, we estimate the term II.
For all $x\in (4R_j)^\complement$, we have
\begin{eqnarray*}
S_L(a_j)(x)
&&\le \lf\{\int_0^{r_{R_j}}\int_{B(x,t)}
|Q_{t^m}(a_j)(y)|^2\,\frac{dydt}{t^{n+1}}\r\}^{1/2}+
\lf\{\int_{r_{R_j}}^\fz\int_{B(x,t)}\cdots\r\}^{1/2}\\
&&=:{\rm II}_1(x)+{\rm II}_2(x).
\end{eqnarray*}
Observe that, when $x\in (4R_j)^\complement$, $|x-y|<t$ and
$z\in R_j:=Q(x_{R_j},r_{R_j})$ for some $x_{R_j}\in\rn$ and $r_{R_j}\in(0,\fz)$,
we see that
$$t+|y-z|\ge |x-z|\ge\frac12|x-x_{R_j}|.$$
By this, \eqref{5.2x1} and the H\"older inequality, we find that,
for all $x\in (4R_j)^\complement$,
\begin{eqnarray}\label{5.3x4}
{\rm II}_1(x)
&&\ls \lf\{\int_0^{r_{R_j}}\int_{|x-y|<t}\lf[\int_{R_j}
\frac{t^\delta}{(t+|y-z|)^{n+\delta}}|a_j(z)|\,dz\r]^2\,\frac{dydt}{t^{n+1}}
\r\}^{\frac12}\\
&&\ls\frac{(r_{R_j})^\delta}{|x-x_{R_j}|^{n+\delta}}
\|a_j\|_{L^2(R_j)}|R_j|^{\frac12}
\ls\frac{(r_{R_j})^\delta}{|x-x_{R_j}|^{n+\delta}}
\frac{|R_j|}{\|\chi_{R_j}\|_{\vlp}}.\noz
\end{eqnarray}
Choose $\delta\in(n[\frac 1{p_-}-1],1)$ and $r\in (0,p_-)$
such that $n+\delta>\frac nr$.
Then, from \eqref{5.3x3x}, \eqref{5.3x4}, Lemma \ref{l-hlmo} and the fact that,
for all $k\in\nn$,
$\chi_{4^kR_j}\le 2^{k\frac nr}[\cm(\chi_{R_j})]^{1/r}$, we deduce that
\begin{eqnarray}\label{5.3x5}
&&\lf\|\sum_{j\in\nn}|\lz_j|{\rm II}_1(\cdot)\chi_{(4R_j)^\complement}\r\|_{\vlp}\\
&&\hs\ls\lf\|\sum_{k\in\nn}\sum_{j\in\nn}\frac{|\lz_j||R_j|}{\|\chi_{R_j}\|_{\vlp}}
\frac{(r_{R_j})^\delta}{|\cdot-x_{R_j}|^{n+\delta}}
\chi_{(4^{k}R_j)\setminus(4^{k-1}R_j)}\r\|_{\vlp}\noz\\
&&\hs\ls\lf\{\sum_{k\in\nn}4^{-k(n+\delta-\frac nr)}\lf\|\sum_{j\in\nn}
\lf(\frac{|\lz_j|}{\|\chi_{R_j}\|_{\vlp}}[\cm(\chi_{R_j})]^{\frac1r}
\r)^{\underline{p}}\r\|_{L^{\frac{p(\cdot)}{\underline{p}}}(\rn)}
\r\}^{\frac 1{\underline{p}}}\noz\\
&&\hs\ls\lf\|\lf\{\sum_{j\in\nn}
\lf[\cm\lf(\frac{|\lz_j|^r}{\|\chi_{R_j}\|_{\vlp}^r}\chi_{R_j}\r)
\r]^{\underline{p}/r}\r\}^{r/\underline{p}}\r\|_{L^{\frac{p(\cdot)}r}(\rn)}^r\noz\\
&&\hs\ls\lf\|\lf\{\sum_{j\in\nn}\lf[\frac{|\lz_j|}{\|\chi_{R_j}\|_{\vlp}}\chi_{R_j}
\r]^{\underline{p}}\r\}^{\frac1{\underline{p}}}\r\|_{\vlp}
\ls\|f\|_{H_{\rm at}^{p(\cdot),2}(\rn)}.\noz
\end{eqnarray}
On the other hand, by \eqref{5.2x2}, \eqref{5.3x3} and the vanishing
moment condition of $a_j$, we obtain
\begin{eqnarray*}
{\rm II}_2(x)
&&\le\lf\{\int_{r_{R_j}}^\fz\int_{|y-x|<t}\lf[\int_{R_j}
|q_{t^m}(y,z)-q_{t^m}(y,x_{R_j})||a_j(z)|\,dz\r]^2\frac{dydt}{t^{n+1}}
\r\}^{\frac12}\\
&&\ls\lf\{\int_{r_{R_j}}^\fz\int_{|y-x|<t}\lf[\int_{R_j}
\frac{|z-x_{R_j}|^\gz t^\delta}{(t+|y-z|)^{n+\delta+\gz}}|a_j(z)|\,dz\r]^2\frac{dydt}{t^{n+1}}
\r\}^{\frac12}\\
&&\ls \frac{(r_{R_j})^\gz}{|\cdot-x_{R_j}|^{n+\gz-\gz_1}}
\lf\{\int_{r_{R_j}}^\fz\lf[\int_{R_j}|a_j(z)|\,dz\r]^2\,t^{-2\gz_1}\frac{dt}t\r\}^{1/2},
\end{eqnarray*}
where $\gz_1\in(0,\gz)$ such that $\gamma-\gz_1\in(n[\frac1{p_-}-1],1)$,
which, together with the H\"older inequality, implies that,
for all $x\in (4R_j)^\complement$,
\begin{eqnarray*}
{\rm II}_2(x)
\ls \frac{(r_{R_j})^{\gz-\gz_1}}{|x-x_{R_j}|^{n+\gz-\gz_1}}
\frac{|R_j|}{\|\chi_{R_j}\|_{\vlp}}.
\end{eqnarray*}
By this and an argument similar to that used in the proof of \eqref{5.3x5},
we conclude that
\begin{eqnarray*}
\lf\|\sum_{j\in\nn}|\lz_j|{\rm II}_2(\cdot)\chi_{(4R_j)^\complement}\r\|_{\vlp}
\ls\|f\|_{H_{\rm at}^{p(\cdot),2}(\rn)}.
\end{eqnarray*}
This, combined with \eqref{5.3x5}, shows that
${\rm II}\ls \|f\|_{H_{\rm at}^{p(\cdot),2}(\rn)}$.
Therefore, $f\in H_L^{p(\cdot)}(\rn)$ and
$$\|f\|_{H_L^{p(\cdot)}(\rn)}=\|S_L(f)\|_{\vlp}\ls \|f\|_{H_{\rm at}^{p(\cdot),2}(\rn)},$$
which further implies that
$H_{\rm at}^{p(\cdot),2}(\rn)\st H_L^{p(\cdot)}(\rn)$.

Conversely, we prove that $H_L^{p(\cdot)}(\rn)\st H_{\rm at}^{p(\cdot),2}(\rn)$.
Let $\alpha$ be a $(p(\cdot),s,L)$-molecule and $\alpha=\pi_L(a)$, where
$a$ is a $(p(\cdot),\fz)$-atom supported on $\wh R$ for some cube $R\st\rn$.
Let $R:=Q(x_R,r_R)$ with $x_R\in\rn$ and $r_R\in(0,\fz)$,
$D_0(R):=2R$ and, when $k\in\nn$,
$D_k(R):=(2^{k+1}R)\setminus (2^kR)$. Moreover, for any $k\in\zz_+$,
we let
$l_k:=\int_{D_k(R)}\alpha(x)\,dx$ and
$$h_k:=\alpha\chi_{D_k(R)}-\frac{\chi_{D_k(R)}}{|D_k(R)|}\int_{D_k(R)}\alpha(x)\,dx.$$
Then, for all $x\in\rn$, we have
\begin{eqnarray*}
\alpha(x)
&&=\sum_{k\in\zz_+}h_k(x)+\sum_{k\in\zz_+}\frac{l_k}{|D_k(R)|}\chi_{D_k(R)}(x)\\
&&=\sum_{k\in\zz_+}h_k(x)+\sum_{k\in\zz_+}N_{k+1}\lf[\wz\chi_{k+1}(x)
-\wz\chi_k(x)\r]=:{\rm J_1}+{\rm J_2},\noz
\end{eqnarray*}
where, for any $k\in\zz_+$, $N_k:=\sum_{j=k}^\fz l_j$ and
$\wz\chi_k:=\frac{\chi_{D_{k}(R)}}{|D_{k}(R)|}$.

We first deal with ${\rm J}_1$.
Obviously, for all $k\in\zz_+$, $\supp h_k\subset 2^{k+1}R$ and
$\int_\rn h_k(x)\,dx=0$.
Moreover, by the H\"older inequality and Proposition \ref{p-tent2}(i), we see that
$$\|h_0\|_{L^2(\rn)}\ls \|\alpha\|_{L^2(\rn)}\ls\|a\|_{T_2^2(\urn)}
\ls |R|^{-1/2}\|\chi_{R}\|_{\vlp}.$$
Since $\supp a\st \wh R$, it follows that, for all $x\in \rn$,
\begin{eqnarray*}
|\alpha(x)|&&\ls \int_0^{r_R}|Q_{s,t^m}(I-P_{s_0,t^m})(a(\cdot,t))(x)|\,\frac{dt}t\\
&&\ls \int_0^{r_R}|Q_{s,t^m}(a(\cdot,t))(x)|\,\frac{dt}t
+\int_0^{r_R}|(Q_{s,t^m}P_{s_0,t^m})(a(\cdot,t))(x)|\,\frac{dt}t,
\end{eqnarray*}
where $s_0$ is as in \eqref{3.6x} and $s\in[s_0,\fz)$.
By \eqref{5.2x1}, \eqref{partial1} and the H\"older inequality, we find that, for all
$k\in\nn$ and $x\in D_k(R)$,
\begin{eqnarray}\label{5.3x6}
\int_0^{r_R}|Q_{s,t^m}(a(\cdot,t))(x)|\,\frac{dt}t
&&\ls\int_0^{r_R}\int_R\frac{t^\delta}{(t+|x-y|)^{n+\delta}}|a(y,t)|\,\frac{dydt}{t}\\
&&\ls\frac{(r_R)^{\frac n2+\delta}}{|x-x_R|^{n+\delta}}\|a\|_{T_2^2(\urn)}
\ls \frac{2^{-k(n+\delta)}}{\|\chi_R\|_{\vlp}}.\noz
\end{eqnarray}
By an argument similar to that used in the proof of \eqref{carleson-6x}, we also have
$$\int_0^{r_R}|(Q_{s,t^m}P_{s_0,t^m})(a(\cdot,t))(x)|\,\frac{dt}t
\ls \frac{2^{-k(n+\delta)}}{\|\chi_R\|_{\vlp}},$$
which, combined with \eqref{5.3x6}, implies that, for all $k\in\nn$ and $x\in D_k(R)$,
\begin{equation}\label{5.4x0}
|\alpha(x)|\ls \frac{2^{-k(n+\delta)}}{\|\chi_R\|_{\vlp}}.
\end{equation}
From this, together with Lemma \ref{l-bigsball}, it follows that
\begin{eqnarray}\label{5.4x1}
\|h_k\|_{L^2(\rn)}
&&\ls\|\alpha\|_{L^2(D_k(R))}
\ls 2^{-k(\frac n2+\delta)}\|\chi_R\|_{\vlp}^{-1}|R|^{\frac12}\\
&&\ls 2^{-k(n+\delta-\frac n{p_-})}\frac{|2^{k+1}R|^{\frac12}}
{\|2^{k+1}R\|_{\vlp}}\noz.
\end{eqnarray}
Thus, for each $k\in\nn$, $2^{k(n+\delta-\frac n{p_-})}h_k$ is a $(p(\cdot),2,0)$-atom
up to a positive constant multiple.
By \eqref{5.4x1} and the fact that $\delta>n(\frac1{p_-}-1)$, we find that
\begin{eqnarray*}
\lf\|\sum_{k=0}^\fz h_k\r\|_{L^2(\rn)}
\ls \sum_{k=0}^\fz\lf\| h_k\r\|_{L^2(\rn)}
\ls\|\chi_R\|_{\vlp}^{-1}|R|^{\frac12},
\end{eqnarray*}
which implies that $\sum_{k=0}^\fz h_k$ converges in $\cs'(\rn)$.
Moreover, from Remark \ref{r-vlp}(i) and the Fatou lemma of $\vlp$
(see \cite[Theorem 2.61]{cfbook}),
we deduce that
\begin{eqnarray}\label{5.4x1y}
&&\lf\|\lf\{\sum_{k=0}^\fz\lf[\frac{2^{-k(n+\delta-n/p_-)}\chi_{2^{k+1}R}}
{\|\chi_{2^{k+1}R}\|_{\vlp}}\r]^{\underline{p}}\r\}^{\frac1{\underline{p}}}\r\|_{\vlp}\\
&&\hs\ls\lf\{\sum_{k=0}^\fz\lf\|\lf[\frac{2^{-k(n+\delta-n/p_-)}\chi_{2^{k+1}R}}
{\|\chi_{2^{k+1}R}\|_{\vlp}}\r]^{\underline{p}}
\r\|_{L^{\frac{p(\cdot)}{\underline{p}}}}\r\}^{\frac1{\underline{p}}}\ls1.\noz
\end{eqnarray}
Therefore,
${\rm J_1}=\sum_{k=0}^\fz h_k\in H_{\rm at}^{p(\cdot),2}(\rn)$.

Next, we consider the term ${\rm J_2}$.
Obviously, for any $k\in\zz_+$,
$$\supp N_{k+1}(\wz\chi_{k+1}-\wz\chi_k)\subset 2^{k+1}R$$
and
$$\int_\rn N_{k+1}\lf[\wz\chi_{k+1}(x)-\wz\chi_k(x)\r]\,dx
=N_{k+1}\int_\rn \lf[\wz\chi_{k+1}(x)-\wz\chi_k(x)\r]\,dx=0.$$
On the other hand, by
\eqref{5.4x0} and Lemma \ref{l-bigsball}, we know that, for each $k\in\zz_+$,
\begin{eqnarray}\label{5.4x2}
\lf\|N_{k+1}(\wz\chi_{k+1}-\wz\chi_k)\r\|_{L^2(\rn)}
&&\ls \frac1{|2^kR|}\int_{(2^{k+1}R)^\complement}|\alpha(x)|\,dx\\
&&\ls 2^{-k(\frac n2+\delta)}\frac{|R|^{\frac 12}}{\|\chi_R\|_{\vlp}}\noz\\
&&\ls 2^{-k(n+\delta-\frac n{p_-})}\frac{|2^{k+1}R|^{\frac12}}
{\|2^{k+1}R\|_{\vlp}}.\noz
\end{eqnarray}
Therefore, for each $k\in\zz_+$, $2^{k(n+\delta-n/p_-)}N_{k+1}(\wz\chi_{k+1}-\wz\chi_k)$
is a $(p(\cdot),2,0)$-atom up to a positive constant multiple.
Moreover, from \eqref{5.4x2}, we deduce that
\begin{eqnarray*}
\lf\|\sum_{k\in\zz_+}N_{k+1}(\wz\chi_{k+1}-\wz\chi_k)\r\|_{L^2(\rn)}
\ls \|\chi_R\|_{\vlp}^{-1}|R|^{\frac12}
\end{eqnarray*}
and hence ${\rm J_2}=\sum_{k\in\zz_+}N_{k+1}(\wz\chi_{k+1}-\wz\chi_k)$
converges in $\cs'(\rn)$.
By this and \eqref{5.4x1y}, we conclude that
${\rm J_2}\in H_{\rm at}^{p(\cdot),2}(\rn)$.
Therefore, for the molecule $\az$, we have
\begin{equation}\label{5.4x3}
\az=\sum_{k\in\zz_+}\frac1{2^{k(n+\delta-n/p_-)}}\wz h_k+
\sum_{k\in\zz_+}\frac1{2^{k(n+\delta-n/p_-)}}\wz N_{k+1}\lf(\wz\chi_{k+1}
-\wz\chi_k\r)
\end{equation}
in $L^2(\rn)$ and hence in $\cs'(\rn)$,
where, for every $k\in\zz_+$, $\wz h_k$ and $\wz N_{k+1}(\wz\chi_{k+1}
-\wz\chi_k)$ are $(p(\cdot),2,0)$-atoms up to a positive constant multiple and,
moreover, $\az\in H_{\rm at}^{p(\cdot),2}(\rn)$.

Now, for all $f\in H_L^{p(\cdot)}(\rn)\cap L^2(\rn)$, by Theorem \ref{t-mol}(i),
we find that $f$ has an atomic decomposition
$f=\sum_{j\in\nn}\lz_j\alpha_j$,
where the summation converges in $L^2(\rn)$ and also in $H_L^{p(\cdot)}(\rn)$,
$\{\lz_j\}_{j\in\nn}\st\cc$ and $\{\alpha_j\}_{j\in\nn}$ are $(p(\cdot),s,L)$-molecules,
as in Definition \ref{d-m}, such that
$$\cb(\{\lz_j\az_j\}_{j\in\nn})\ls \|f\|_{H_L^{p(\cdot)}(\rn)}.$$
Moreover, we may assume that, for each $j\in\nn$,
 $\alpha_j$ is a molecule associated with some cube $R_j:=Q(x_j,r_j)$ for $x_j\in\rn$ and
$r_j\in(0,\fz)$.
Then, from what we have proved as in \eqref{5.4x3}, we deduce that
\begin{eqnarray*}
f=\sum_{j\in\nn}\sum_{k\in\zz_+}\frac{\lz_j}{2^{k(n+\delta-\frac n{p_-})}}\wz h_{j,k}+
\sum_{j\in\nn}\sum_{k\in\zz_+}\frac{\lz_j}{2^{k(n+\delta-\frac n{p_-})}}
\wz N_{j,k+1}\lf(\wz\chi_{j,k+1}-\wz\chi_{j,k}\r)
\end{eqnarray*}
in $L^2(\rn)$ and hence also in $\cs'(\rn)$,
where, for each $j\in\nn$ and $k\in\zz_+$,
$\wz h_{j,k}$ and $\wz N_{j,k+1}(\wz\chi_{j,k+1}-\wz\chi_{j,k})$
are $(p(\cdot),2,0)$-atoms, supported on $2^{k+1}R_j$, up to a positive
constant multiple. On the other hand,
by Lemma \ref{l-bigsball}, we see that, for any $j\in\nn$ and $k\in\zz_+$,
\begin{equation*}
\frac{\|\chi_{R_j}\|_{\vlp}}{\|\chi_{2^kR_j}\|_{\vlp}}
\ls\lf(\frac{|R_j|}{|2^{k}R_j|}\r)^{\frac1{p_+}}\sim 2^{-k\frac n{p_+}}.
\end{equation*}
Then, by choosing $\delta\in(0,1)$ such that
$\delta\in (n[\frac 2{p_-}-1-\frac 1{p_+}],1)$,
Lemma \ref{l-hlmo} and the fact that, for any $j\in\nn$, $k\in\zz_+$,
$r\in(0,\underline{p})$ and $x\in\rn$,
$$\chi_{2^{k}R_j}(x)\le 2^{kn/r}\lf[\cm(\chi_{R_j})(x)\r]^{\frac1r},$$
we deduce that
\begin{eqnarray*}
&&\lf\|\lf\{\sum_{j\in\nn}\sum_{k\in\zz_+}
\lf[\frac{|\lz_j|2^{-k(n+\delta-\frac n{p_-})}\chi_{2^{k+1}R_j}}
{\|\chi_{2^{k+1}R_j}\|_{\vlp}}\r]^{\underline{p}}\r\}^{\frac1{\underline{p}}}\r\|_{\vlp}\\
&&\hs\ls\lf\|\lf\{\sum_{j\in\nn}\sum_{k\in\zz_+}
\lf[\frac{|\lz_j|2^{-k(n+\delta+\frac n{p_+}-\frac n{p_-}-\frac nr)}
[\cm(\chi_{R_j})]^{\frac1r}
}{\|\chi_{R_j}\|_{\vlp}}\r]^{\underline{p}}\r\}^{\frac1{\underline{p}}}\r\|_{\vlp}\\
&&\hs\ls\lf\|\lf\{\sum_{j\in\nn}
\lf[\frac{|\lz_j|[\cm(\chi_{R_j})]^{\frac1r}
}{\|\chi_{R_j}\|_{\vlp}}\r]^{\underline{p}}\r\}^{\frac1{\underline{p}}}\r\|_{\vlp}\\
&&\hs\ls \cb(\{\lz_j\az_j\}_{j\in\nn})\ls \|f\|_{H_L^{p(\cdot)}(\rn)}.
\end{eqnarray*}
Therefore, $f\in H_{\rm at}^{p(\cdot),2}(\rn)$ and hence
$H_L^{p(\cdot)}(\rn)\subset H_{\rm at}^{p(\cdot),2}(\rn)$.
This finishes the proof of Theorem \ref{t-5.2x}.
\end{proof}

\begin{remark}
When $p(\cdot)$ is a constant exponent, Theorem \ref{t-5.2x}
goes back to \cite[Theorem 6.1]{yan08} (see also
\cite[Theorem 6.1]{jyz09}). We point out that the proof of
Theorem \ref{t-5.2x} borrows some ideas from the proof of \cite[Theorem 6.1]{jyz09}.
\end{remark}

\subsection{Fractional integrals $L^{-\gamma}$ on spaces $H_L^{p(\cdot)}(\rn)$}

\hskip\parindent
Let $L$ satisfy Assumptions \ref{as-a} and \ref{as-b} as in Subsection \ref{s2.2}.
In this subsection, we establish the boundedness of the fractional integral
on variable exponent Hardy spaces associated with the operator $L$.
Recall that, for any $\gamma\in(0,\frac nm)$ with $m$ as in Assumption \ref{as-a}, the generalized fractional
integral $L^{-\gamma}$ associated with $L$ is defined by setting,
for all $f\in L^2(\rn)$ and $x\in\rn$,
\begin{equation*}
L^{-\gamma}(f)(x):=\frac1{\Gamma(\gamma)}\int_0^\fz t^{\gamma-1}e^{-tL}(f)(x)\,dt,
\end{equation*}
where $\Gamma(\gz)$ is an appropriate positive constant; see \cite[p.\,4400]{yan08}.
Notice that, if $L:=-\Delta$ with $\Delta$ being the Laplacian, then
$L^{-\gz}$ becomes the classical fractional integral; see, for example,
\cite[Chapter 5]{stein70}. We also point out
that the Hardy-Littlewood-Sobolev inequality related to the semigroup itself of an operator was studied
by Yoshikawa \cite{Yosh70}.

\begin{remark}\label{r-frac}
Let $L$ satisfy Assumptions \ref{as-a} and \ref{as-b}.
For $\gamma\in(0,\frac nm)$ with $m$ as in Assumption \ref{as-a}, define the operator $\wz L^{-\gamma}$ by setting,
for all $f\in L^2(\rn)$ and $x\in\rn$,
$$\wz L^{-\gamma}(f)(x):=\frac1{\Gamma(\gamma)}\int_0^\fz t^{\gamma-1}
|e^{-tL}(f)(x)|\,dt.$$
It was proved in \cite[Lemma 5.1(ii)]{jyz09} that, if $\gz\in(0,\frac nm)$ and
$p_1,\,p_2\in(1,\fz)$
satisfy $\frac1{p_2}=\frac1{p_1}-\frac{m\gz}n$, then $\wz L^{-\gz}$
is bounded from $L^{p_1}(\rn)$ into $L^{p_2}(\rn)$.
\end{remark}

The main result of this subsection is stated as follows.

\begin{theorem}\label{t-frac}
Let $L$ satisfy Assumptions \ref{as-a} and \ref{as-b}, $\gz\in(0,\frac nm)$
with $m$ as in Assumption \ref{as-a},
$p(\cdot)\in C^{\log}(\rn)$ satisfy $\frac n{n+\theta(L)}<p_-\le p_+\le1$
with $p_-$, $p_+$ and $\theta(L)$, respectively, as in \eqref{2.1x} and \eqref{2.3x}.
Assume that $q(\cdot)$ is defined by setting, for all $x\in\rn$,
\begin{equation*}
\frac1{q(x)}:=\frac1{p(x)}-\frac {m\gz}n.
\end{equation*}
Then the fractional integral $L^{-\gz}$ maps $H_L^{p(\cdot)}(\rn)$
continuously into $H_L^{q(\cdot)}(\rn)$.
\end{theorem}

To prove Theorem \ref{t-frac}, we need the following technical lemma, which
is just \cite[Lemma 5.2]{Sa13} and plays a key role in the proof of Theorem
\ref{t-frac}.

\begin{lemma}\label{l-frac1}
Let $\delta\in(0,n)$ and $p(\cdot)\in C^{\log}(\rn)$ satisfy $p_+\in(0,\frac n\delta)$.
Assume that $q(\cdot)\in\cp(\rn)$ is defined by setting, for all $x\in\rn$,
$\frac1{q(x)}:=\frac1{p(x)}-\frac {\delta}n$. Then there exists a positive constant
$C$ such that, for all sequences $\{R_j\}_{j\in\nn}$ of cubes of $\rn$ and
$\{\lz_j\}_{j\in\nn}\st\cc$,
$$\lf\|\sum_{j\in\nn}|\lz_j||R_j|^{\frac{\delta}n}\chi_{R_j}\r\|_{L^{q(\cdot)}(\rn)}
\le C \lf\|\sum_{j\in\nn}|\lz_j|\chi_{R_j}\r\|_{L^{p(\cdot)}(\rn)}.$$
\end{lemma}

\begin{proof}[Proof of Theorem \ref{t-frac}]
To prove this theorem, we only need to show that, for all $f\in H_L^{p(\cdot)}(\rn)\cap L^2(\rn)$,
\begin{equation}\label{frac-2}
\|S_L(L^{-\gz}(f))\|_{L^{q(\cdot)}(\rn)}\ls
\|f\|_{H_L^{p(\cdot)}(\rn)},
\end{equation}
since $H_L^{p(\cdot)}(\rn)\cap L^2(\rn)$ is dense in $H_L^{p(\cdot)}(\rn)$.

Let $f\in H_L^{p(\cdot)}(\rn)\cap L^2(\rn)$.
Then, by Theorem \ref{t-mol}(ii), we see that
there exist $\{\lz_j\}_{j\in\nn}\st\cc$ and a sequence $\{\alpha_j\}_{j\in\nn}$
of $(p(\cdot),s_0,L)$-molecules associated with cubes $\{R_j\}_{j\in\nn}$
such that
$f=\sum_{j\in\nn}\lz_j \alpha_j$ in $H_L^{p(\cdot)}(\rn)$ and also in $L^2(\rn)$, and
$$\ca(\{\lz_j\}_{j\in\nn},\{R_j\}_{j\in\nn})\ls\|f\|_{H_L^{p(\cdot)}(\rn)}.$$
Observe that $L^{-\gz}$ is bounded from $L^{2}(\rn)$ to $L^q(\rn)$ for some
$q\in(1,\fz)$ such that $\frac1q=\frac12-\frac{m\gz}n$
(see \cite[Theorem 5.1]{yan08} and also Remark \ref{r-frac}).
It follows that, for almost every $x\in\rn$,
\begin{eqnarray*}
|L^{-\gz}(f)(x)|
\ls\sum_{j\in\nn}|\lz_j|\int_0^\fz t^{\gamma-1}|e^{-tL}(\az_j)(x)|\,dt
=:\sum_{j\in\nn}|\lz_j|\wz L^{-\gamma}(\az_j)(x)
\end{eqnarray*}
and hence
\begin{eqnarray}\label{frac-2x}
\lf\|S_L(L^{-\gamma}(f))\r\|_{L^{q(\cdot)}(\rn)}
&&\ls\lf\|\sum_{j\in\nn}|\lz_j|S_L(\wz L^{-\gz}(\az_j))
\chi_{4R_j}\r\|_{L^{q(\cdot)}(\rn)}\\
&&\quad+
\lf\|\sum_{j\in\nn}|\lz_j|S_L(\wz L^{-\gz}(\az_j))
\chi_{(4R_j)^\complement}\r\|_{L^{q(\cdot)}(\rn)}
=:{\rm I}_1+{\rm I}_2.\noz
\end{eqnarray}

To deal with ${\rm I}_1$, let $r\in(1,2)$. Then,
by the H\"older inequality, \eqref{SL-bounded}, Proposition \ref{p-tent2} and Remark \ref{r-frac},
 we find that
\begin{eqnarray*}
&&\lf\|S_L(\wz L^{-\gz}(\az_j))\chi_{4R_j}\r\|_{L^r(\rn)}\\
&&\hs\ls|R_j|^{\frac1r-\frac1q}\lf\|S_L(\wz L^{-\gz}(\az_j))\chi_{4R_j}\r\|_{L^q(\rn)}
\ls|R_j|^{\frac1r-\frac1q}\lf\|\wz L^{-\gz}(\az_j)\chi_{4R_j}\r\|_{L^q(\rn)}\\
&&\hs \ls|R_j|^{\frac1r-\frac1q}\|\az_j\|_{L^2(\rn)}
\ls[\ell(R_j)]^{m\gz}\frac{|R_j|^{\frac1r}}{\|\chi_{R_j}\|_{\vlp}}.
\end{eqnarray*}
This, combined with Lemmas \ref{l-estimate} and \ref{l-frac1}, implies that
\begin{eqnarray}\label{frac-2y}
{\rm I}_1&&\ls\lf\|\sum_{j\in\nn}
\lf[\frac{|\lz_j|}{\|\chi_{R_j}\|_{\vlp}}\r]^{\underline{q}}
[\ell(R_j)]^{m\gz\underline{q}}\chi_{4R_j}
\r\|_{L^{\frac{q(\cdot)}{\underline{q}}}}^{\frac1{\underline{q}}}\\
&&\ls\lf\|\sum_{j\in\nn}
\lf[\frac{|\lz_j|\chi_{4R_j}}{\|\chi_{R_j}\|_{\vlp}}\r]^{\underline{q}}
\r\|_{L^{\frac{p(\cdot)}{\underline{q}}}}^{\frac1{\underline{q}}}
\ls\lf\|\lf\{\sum_{j\in\nn}\lf[\frac{|\lz_j|\chi_{4R_j}}{\|\chi_{R_j}\|_{\vlp}}
\r]^{\underline{p}}\r\}^{\frac1{\underline{p}}}\r\|_{\vlp}\noz\\
&&\sim \ca(\{\lz_j\}_{j\in\nn},\{R_j\}_{j\in\nn})\ls\|f\|_{H_L^{p(\cdot)}(\rn)},\noz
\end{eqnarray}
where $\underline{q}:=\min\{1,q_-\}$ with $q_-$ as in \eqref{2.1x} via
$p(\cdot)$ replaced by $q(\cdot)$.

Next, we estimate ${\rm I}_2$. Since $\frac n{n+\theta(L)}<p_-$, it follows that
there exists $\vez\in(0,\theta(L))$ such that $\frac n{n+\vez}<p_-$. Moreover,
we may choose $r_0\in(0,p_-)$ such that $\vez\in(n[\frac1{r_0}-1],\theta(L))$.
Assume that, for each $j\in\nn$, $R_j:=Q(x_j,r_j)$ for some $x_j\in\rn$ and
$r_j\in(0,\fz)$. Then, by an argument similar to that used in the proof
of \cite[(5.3)]{yan08}, we conclude that, for all $j\in\nn$ and
$x\in (4R_j)^\complement$,
\begin{equation}\label{frac-3}
S_L(\wz L^{-\gz}(\az_j))(x)\ls \frac{|r_j|^{\vez+m\gz+n}}{|x-x_{j}|^{n+\vez}}
\frac1{\|\chi_{R_j}\|_{\vlp}}.
\end{equation}
For any $k,\,j\in\nn$, let
$D_k(R_j):=(2^{k+2}R_j)\backslash (2^{k+1}R_j)$. Then,
by \eqref{frac-3} and Lemma \ref{l-frac1}, we see that
\begin{eqnarray*}
{\rm I}_2
&&\ls \lf\|\sum_{k,\,j\in\nn}\frac{|\lz_j|}{\|\chi_{R_j}\|_{\vlp}}
\frac{r_j^{\vez+m\gz+n}}{|\cdot-x_j|^{n+\vez}}\chi_{D_k(R_j)}\r\|_{L^{q(\cdot)}(\rn)}\\
&&\ls \lf\|\sum_{k,\,j\in\nn}\frac{|\lz_j|}{\|\chi_{R_j}\|_{\vlp}}
\frac{(2^kr_j)^{m\gz}}{2^{k(n+\vez+m\gz)}}\chi_{2^{k+2}R_j}\r\|_{L^{q(\cdot)}(\rn)}\\
&&\ls \lf\|\sum_{k,\,j\in\nn}\frac{|\lz_j|}{\|\chi_{R_j}\|_{\vlp}}
\frac{\chi_{2^{k+2}R_j}}{2^{k(n+\vez+m\gz)}}\r\|_{L^{p(\cdot)}(\rn)}.
\end{eqnarray*}
Thus, from Lemma \ref{l-hlmo}, Remark \ref{r-vlp}(i) and
the fact that, for any $k\in\nn$,
$$\chi_{2^{k+2}R_j}\ls 2^{kn/r_0}[\cm(\chi_{R_j})]^{1/r_0},$$
we deduce that
\begin{eqnarray*}
{\rm I}_2
&&\ls \lf[\sum_{k\in\nn}\frac{2^{kn\underline{p}/r_0}}
{2^{k\underline{p}(n+\vez+m\gz)}}\lf\|\sum_{j\in\nn}
\lf\{\frac{|\lz_j|}{\|\chi_{R_j}\|_{\vlp}}[\cm(\chi_{R_j})]^\frac1{r_0}
\r\}^{\underline{p}}\r\|_{L^{\frac{p(\cdot)}{\underline{p}}}(\rn)}
\r]^{\frac1{\underline{p}}}\\
&&\ls\lf\|\lf\{\sum_{j\in\nn}\lf[\frac{|\lz_j|\chi_{R_j}}{\|\chi_{R_j}
\|_{\vlp}}\r]^{\underline{p}}
\r\}^{\frac1{\underline{p}}}\r\|_{\vlp}\ls\|f\|_{H_L^{p(\cdot)}(\rn)}.
\end{eqnarray*}
This, together with \eqref{frac-2x} and \eqref{frac-2y}, implies that
\eqref{frac-2} holds true,
which shows that $L^{-\gz}$ is bounded from $H_L^{p(\cdot)}(\rn)$
to $H_L^{q(\cdot)}(\rn)$ and hence completes the proof of Theorem \ref{t-frac}.
\end{proof}

\begin{remark}
In the case of constant exponents, Theorem \ref{t-frac} was obtained by Yan
\cite[Theorem 5.1]{yan08}.
\end{remark}

\emph{Acknowledgements.} The authors would like to express their
deep thanks to the referee for his very careful reading and useful
comments which do improve the presentation of this article.

\bigskip

\noindent  Dachun Yang  and Ciqiang Zhuo (Corresponding author)

\medskip

\noindent  School of Mathematical Sciences, Beijing Normal University,
Laboratory of Mathematics and Complex Systems, Ministry of
Education, Beijing 100875, People's Republic of China

\smallskip

\noindent {\it E-mails}: \texttt{dcyang@bnu.edu.cn} (D. Yang)

\hspace{0.98cm} \texttt{cqzhuo@mail.bnu.edu.cn} (C. Zhuo)

%\hspace{0.98cm}\texttt{wenyuan@bnu.edu.cn} (W. Yuan)

\end{document}